\documentclass[10pt,twoside]{amsart}

\usepackage[
  paper=a4paper,
  headsep=15pt,text={138mm,213mm},centering,includehead
]{geometry}

\usepackage{tikz}
\usetikzlibrary{calc,arrows,cd}

\usepackage{amsfonts,amstext,amsmath,amsthm,amssymb}
\usepackage{float,array,calc,booktabs,enumitem}
\usepackage{multirow,arydshln}

\usepackage{color}
\usepackage{url}
\usepackage[colorlinks,
  linkcolor={red!50!black},
  citecolor={blue!50!black},
  urlcolor={blue!80!black}
]{hyperref}

\iftrue
\makeatletter
\def\@settitle{%
  \vspace*{-25pt}
  \begin{flushleft}%
    \LARGE\bfseries
    \strut\@title\strut
  \end{flushleft}%
}
\def\@setauthors{%
  \begingroup
  \def\thanks{\protect\thanks@warning}%
  \trivlist
  \raggedright
  \large \@topsep27\p@\relax
  \advance\@topsep by -\baselineskip
  \item\relax
  \author@andify\authors
  \def\\{\protect\linebreak}%
  \authors
  \ifx\@empty\contribs
  \else
    ,\penalty-3 \space \@setcontribs
    \@closetoccontribs
  \fi
  \normalfont
  \endtrivlist
  \endgroup
}
\def\@setaddresses{\par
  \nobreak \begingroup
  \small\raggedright
  \def\author##1{\nobreak\addvspace\smallskipamount}%
  \def\\{\unskip, \ignorespaces}%
  \interlinepenalty\@M
  \def\address##1##2{\begingroup
    \par\addvspace\bigskipamount\noindent
    \@ifnotempty{##1}{(\ignorespaces##1\unskip) }%
    {\ignorespaces##2}\par\endgroup}%
  \def\curraddr##1##2{\begingroup
    \@ifnotempty{##2}{\nobreak\noindent\curraddrname
      \@ifnotempty{##1}{, \ignorespaces##1\unskip}\/:\space
      ##2\par}\endgroup}%
  \def\email##1##2{\begingroup
    \@ifnotempty{##2}{\nobreak\noindent E-mail address%
      \@ifnotempty{##1}{, \ignorespaces##1\unskip}\/:\space
      \ttfamily##2\par}\endgroup}%
  \def\urladdr##1##2{\begingroup
    \def~{\char`\~}%
    \@ifnotempty{##2}{\nobreak\noindent\urladdrname
      \@ifnotempty{##1}{, \ignorespaces##1\unskip}\/:\space
      \ttfamily##2\par}\endgroup}%
  \addresses
  \endgroup
  \global\let\addresses=\@empty
}
\def\@setabstracta{%
    \ifvoid\abstractbox
  \else
    \skip@17pt \advance\skip@-\lastskip
    \advance\skip@-\baselineskip \vskip\skip@
    \box\abstractbox
    \prevdepth\z@ % because \abstractbox is a vtop
    \vskip-28pt
  \fi
}
\renewenvironment{abstract}{%
  \ifx\maketitle\relax
    \ClassWarning{\@classname}{Abstract should precede
      \protect\maketitle\space in AMS document classes; reported}%
  \fi
  \global\setbox\abstractbox=\vtop \bgroup
    \normalfont\small
    \list{}{\labelwidth\z@
      \leftmargin0pc \rightmargin\leftmargin
      \listparindent\normalparindent \itemindent\z@
      \parsep\z@ \@plus\p@
      
    }%
    \item[\hskip\labelsep\bfseries\abstractname.]%
}{%
  \endlist\egroup
  \ifx\@setabstract\relax \@setabstracta \fi
}

\def\ps@headings{\ps@empty
  \def\@evenhead{%
    \setTrue{runhead}%
    \normalfont\scriptsize
    \rlap{\thepage}\hfill
    \def\thanks{\protect\thanks@warning}%
    \leftmark{}{}}%
  \def\@oddhead{%
    \setTrue{runhead}%
    \normalfont\scriptsize
    \def\thanks{\protect\thanks@warning}%
    \rightmark{}{}\hfill \llap{\thepage}}%
  \let\@mkboth\markboth
}\ps@headings

\def\section{\@startsection{section}{1}%
  \z@{-1.4\linespacing\@plus-.5\linespacing}{.8\linespacing}%
  {\normalfont\bfseries\Large}}
\def\subsection{\@startsection{subsection}{2}%
  \z@{-.8\linespacing\@plus-.3\linespacing}{.5\linespacing\@plus.2\linespacing}%
  {\normalfont\bfseries\large}}
\def\subsubsection{\@startsection{subsubsection}{3}%
  \z@{.7\linespacing\@plus.2\linespacing}{-1.5ex}%
  {\normalfont\bfseries}}
\def\@secnumfont{\bfseries}

\renewcommand\contentsnamefont{\bfseries}
\def\@starttoc#1#2{\begingroup
  \setTrue{#1}%
  \par\removelastskip\vskip\z@skip
  \@startsection{}\@M\z@{\linespacing\@plus\linespacing}%
    {.5\linespacing}{%\centering
      \contentsnamefont}{#2}%
  \ifx\contentsname#2%
  \else \addcontentsline{toc}{section}{#2}\fi
  \makeatletter
  \@input{\jobname.#1}%
  \if@filesw
    \@xp\newwrite\csname tf@#1\endcsname
    \immediate\@xp\openout\csname tf@#1\endcsname \jobname.#1\relax
  \fi
  \global\@nobreakfalse \endgroup
  \addvspace{32\p@\@plus14\p@}%
  \let\tableofcontents\relax
}
\def\contentsname{Contents}
\def\l@section{\@tocline{2}{.5ex}{0mm}{5pc}{}}
\def\l@subsection{\@tocline{2}{0pt}{2em}{5pc}{}}
\makeatother
\fi %\iftrue/false for amsart.cls hack

\makeatletter
\def\Nopagebreak{\@nobreaktrue\nopagebreak}
\makeatother

\renewcommand{\setminus}{{\smallsetminus}}

\newcommand{\bp}{\begin{pmatrix}}
\newcommand{\ep}{\end{pmatrix}}

\theoremstyle{plain}
\newtheorem{theorem}[equation]{Theorem}
\newtheorem{lemma}[equation]{Lemma}
\newtheorem{proposition}[equation]{Proposition}
\newtheorem*{proposition*}{Proposition}

\newtheorem{corollary}[equation]{Corollary}

\newtheorem{theoremalpha}{Theorem}

\newtheoremstyle{theorem-giventitle}
        {}{}              %%% space between body and thm
        {\itshape}                      %%% Thm body font
        {}                              %%% Indent amount (empty = no indent)
        {\bfseries}                     %%% Thm head font
        {.}                             %%% Punctuation after thm head
        { }                             %%% Space after thm head
        {\thmnote{\bfseries#3}}%%% Thm head spec
\theoremstyle{theorem-giventitle}
\newtheorem{theorem-named}{}

\theoremstyle{definition}

\newtheorem{definition}[equation]{Definition}

\newtheorem{remark}[equation]{Remark}

\newtheorem{claim}{Claim}
\numberwithin{claim}{subsection}

\theoremstyle{remark}

\newtheoremstyle{definition-giventitle}
        {}{}              %%% space between body and thm
        {}                      %%% Thm body font
        {}                              %%% Indent amount (empty = no indent)
        {\bfseries}                     %%% Thm head font
        {.}                             %%% Punctuation after thm head
        {.7em}                             %%% Space after thm head
        {\thmnote{\bfseries#3}}%%% Thm head spec
\theoremstyle{definition-giventitle}
\newtheorem{step-named}{}
\newtheorem{claim-named}{}

\numberwithin{equation}{section}

\makeatletter
  \def\rddots{\mathinner{\mkern1mu\raise\p@
     \vbox{\kern7\p@\hbox{.}}\mkern2mu
     \raise4\p@\hbox{.}\mkern2mu\raise7\p@\hbox{.}\mkern1mu}}
\makeatother

\def\Z{\mathbb Z}

\def\Q{\mathbb Q}

\def\sm{\setminus}
\def\d{\partial}

\def\cupover#1{\mathbin{\mathop{\cup}_{#1}}}

\DeclareMathOperator\lk{lk}
\DeclareMathOperator\Wh{Wh}

\def\diag{\operatorname{diag}}

\def\sbmatrix#1{\big[\begin{smallmatrix}#1\end{smallmatrix}\big]}

\begin{document}

\title{A family of freely slice good boundary links}

\author{Jae Choon Cha}
\address{Department of Mathematics, POSTECH,
  Pohang Gyeongbuk 37673, Republic of Korea
  \linebreak
  School of Mathematics, Korea Institute for Advanced Study,
  Seoul 02455, Republic of Korea}
\email{jccha@postech.ac.kr}

\author{Min Hoon Kim}
\address{School of Mathematics, Korea Institute for Advanced Study,
  Seoul 02455, Republic of Korea}
\email{kminhoon@kias.re.kr}

\author{Mark Powell}
\address{Department of Mathematical Sciences, Durham University,
  United Kingdom
}
\email{mark.a.powell@durham.ac.uk}

\def\subjclassname{\textup{2010} Mathematics Subject Classification}
\expandafter\let\csname subjclassname@1991\endcsname=\subjclassname
\expandafter\let\csname subjclassname@2000\endcsname=\subjclassname

\begin{abstract}
  We show that every good boundary link with a pair of derivative links on a
  Seifert surface satisfying a homotopically trivial plus assumption is
  freely slice.  This subsumes all previously known methods for freely slicing
  good boundary links with two or more components, and provides new freely
  slice links.
\end{abstract}

\maketitle

\section{Introduction}

The still open topological surgery conjecture for $4$-manifolds is equivalent to
the statement that all good boundary links are freely
slice~\cite[p.~243]{Freedman-Quinn:1990-1}. A \emph{good boundary link} is an
$m$-component link $L$ that admits a homomorphism $\phi\colon \pi_1(S^3 \sm L)
\to F$ sending the meridians to the $m$ generators of the free group $F$ of
rank~$m$, such that the kernel of $\phi$ is perfect, or equivalently such that
$H_1(S^3 \sm L;\Z F) = 0$. A Whitehead double of a link with vanishing linking
numbers is a special case.  We say that a link $L$ is \emph{freely slice} if $L$
bounds slicing discs in $D^4$ whose complement has free fundamental group
generated by the meridians of~$L$.

The goal of this paper is to describe a family of freely slice good boundary
links.  Previous work, including \cite{Freedman:1982-2, Freedman:1985-1,
Freedman:1988-2, Freedman:1993-1, Freedman-Teichner:1995-2}, has often focussed
on Whitehead doubles of links.  A complete list with a detailed discussion can
be found in Section~\ref{section:previous-results-comparison}. We expand these
results to a much larger class of good boundary links. After stating our main
theorem, we will explain the terminology used therein.

\begin{theoremalpha}
  \label{theorem:main}
  Suppose that $L$ is a good boundary link that has a Seifert surface admitting
  a homotopically trivial\/$^+$ good basis. Then $L$ is freely slice.
\end{theoremalpha}

We need to introduce the notion of a \emph{good basis} on a boundary link
Seifert surface, and what it means for a good basis to be \emph{homotopically
trivial\/$^+$}. Consider a collection $V$ of disjoint Seifert surfaces for a
boundary link~$L$. We call such a collection a \emph{boundary link Seifert
surface}.  The \emph{Seifert form} $\theta \colon H_1(V;\Z) \times H_1(V;\Z) \to
\Z$ is defined via linking numbers of curves on $V$ and push-offs of such curves
into $S^3 \sm V$, in direct analogy with the knot case.  Choosing a basis for
$H_1(V;\Z)$ gives rise to a \emph{Seifert matrix} representing $\theta$.  We
give a detailed description in
Section~\ref{subsection:seifert-matrix-s-equivalence}.

\begin{definition}
  \label{definition:good-basis}
  A \emph{good basis} is a collection $\{a_i,b_i\}$ of simple closed curves on
  $V$, representing a basis for $H_1(V;\Z)$, whose geometric intersections are
  $a_i \cdot a_j = 0$, $b_i \cdot b_j = 0$, $a_i\cdot b_j = \delta_{ij}$, and
  such that the Seifert matrix of $V$ with respect to the basis is reducible by
  a sequence of elementary $S$-reductions to the null matrix.  That is, after
  reordering if necessary, the Seifert matrix with respect to
  $\{a_1,b_1,a_2,b_2,\ldots\}$ is of the form
  \begin{equation}
    \left[\begin{array}{cc;{.5pt/1pt}cc;{.5pt/1pt}c;{.5pt/1pt}cc}
      0 & \varepsilon_1   &  0 & *  & \multirow{2}{*}{${}\cdots{}$} &  0 & * \\
      1-\varepsilon_1 & 0  &  0 & *  &  &  0 & * \\ \hdashline[.5pt/1pt]
      0 & 0  &  0 & \varepsilon_2   & \multirow{2}{*}{${}\cdots{}$} &  0 & * \\
      * & *  &  1-\varepsilon_2 & 0  &  &  0 & * \\ \hdashline[.5pt/1pt]
      \multicolumn{2}{c;{.5pt/1pt}}{\vdots} & \multicolumn{2}{c;{.5pt/1pt}}{\vdots}
        & \ddots & \multicolumn{2}{c}{\vdots} \\[.5ex] \hdashline[.5pt/1pt]
      0 & 0  &  0 & 0  & \multirow{2}{*}{${}\cdots{}$} &  0 & \varepsilon_r \\
      * & *  &  * & *  &  &  1-\varepsilon_r & 0 \\
    \end{array}\right]
    \label{equation:s-reducible-matrix}
  \end{equation}
  where $\varepsilon_i=0$ or $1$ for each $i$, and each $*$ designates an
  arbitrary integer entry.
\end{definition}

Note that (\ref{equation:s-reducible-matrix}) contains no information about the
connected component of $V$ on which the pair $(a_i,b_i)$ lies. We prove the
following fact in Section~\ref{section:link-module-seifert-matrix}.

\begin{theorem-named}
  [Corollary~\ref{corollary:good-bd-link-iff-good-basis-exists}]

  A boundary link is good if and only if it has a boundary link Seifert surface
  admitting a good basis.
\end{theorem-named}

Recall~\cite{Milnor:1954-1} that an $m$-component link is said to be
\emph{homotopically trivial} if it is homotopic to the trivial link through
immersions of $m$ copies of the circle into $S^3$ with disjoint images. The
strongest result in the literature on freely slicing good boundary links, due to
Freedman and Teichner~\cite{Freedman-Teichner:1995-2}, states that a Whitehead
double of a homotopically trivial$^+$ \emph{link} is freely slice. Here a link
$J$ is said to be \emph{homotopically trivial\/$^+$} if for each component $J_i$
of $J$, the link $J\cup J_i^+$ is homotopically trivial, where $J_i^+$ is a zero
linking parallel copy of~$J_i$.

Our notion of a \emph{homotopically trivial$^+$ good basis} extends the
Freedman-Teichner definition for links to a more general context.   Let
$\{a_i,b_i\}$ be a good basis on a boundary link Seifert surface~$V$. Inspecting
the diagonal blocks of~\eqref{equation:s-reducible-matrix}, it follows that
~$a_i$ has trivial linking number with exactly one of the two translates of
~$b_i$ in the positive and negative normal directions of~$V$. Call this
push-off~$b_i'$ and let $K := \bigsqcup b_i'$.

\begin{definition}
  \label{definition:good-basis-HT+}
  We call the good basis $\{a_i,b_i\}$ \emph{homotopically trivial\/$^+$} if
  both of the links $K\sqcup a_j$ and $K\sqcup b_j$ are homotopically trivial
  for every~$j$.
\end{definition}

Observe that if $\{a_i,b_i\}$ is a homotopically trivial$^+$ good basis then
each of the sets of curves $\{a_i\}$ and $\{b_i\}$ is a \emph{derivative} of
$L$, that is a collection of mutually disjoint curves on a Seifert surface for
$L$ whose linking and self-linking numbers all vanish.  In fact all the entries
of (\ref{equation:s-reducible-matrix}) indicated by a $*$ are $0$ for a
homotopically trivial$^+$ good basis.   Boundary links admitting this kind of
Seifert matrix were termed \emph{generalised doubles} by Freedman and Krushkal
in~\cite[Definition~5.9]{Freedman-Krushkal-2016-A}.

To understand the connection between the two notions of homotopically
trivial$^+$, consider $\Wh(J)$, an untwisted Whitehead double of the link $J$
with arbitrary clasp signs. Recall that $\Wh(J)$ bounds a standard boundary link
Seifert surface with each connected component genus one.
Figure~\ref{figure:good-basis-for-wh} illustrates the case of a positive clasp.
It is straightforward to verify the following.
\begin{enumerate}[label=(\roman*)]
\item If $J$ has vanishing pairwise linking numbers, so that $\Wh(J)$ is
a good boundary link, then the curves $a_i$ and $b_i$ shown in
Figure~\ref{figure:good-basis-for-wh} form a good basis.
\item The good basis $\{a_i,b_i\}$ is homotopically trivial$^+$ if and only if
the link $J$ is homotopically trivial$^+$ in the sense of Freedman-Teichner. We
provide the details in Lemma~\ref{lemma:wh-and-ht+}.
\end{enumerate}

\begin{figure}[htb!]
  \includegraphics{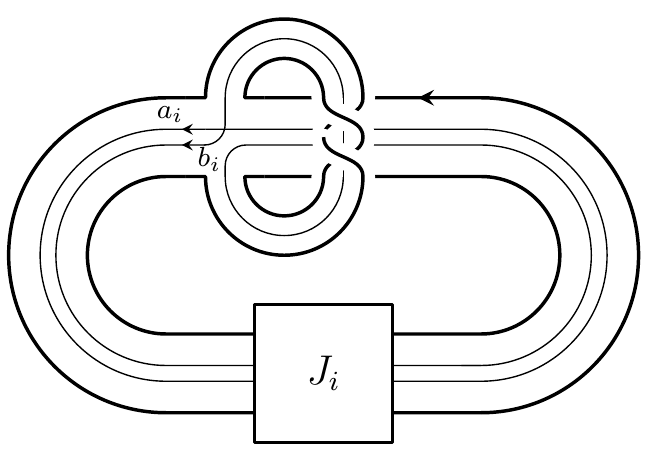}
  \caption{Standard good basis for a Whitehead double.}
  \label{figure:good-basis-for-wh}
\end{figure}

Thus the fact that Whitehead doubles of homotopically trivial$^+$ links are
slice is a special case of Theorem~\ref{theorem:main}. Moreover,
Theorem~\ref{theorem:main} provides new freely slice links, not previously known
to be slice.  In particular, in Section~\ref{subsection:examples}, we provide
examples not arising as a Whitehead double.

We close the introduction by noting that it is logically conceivable that
Theorem~\ref{theorem:main} applies to freely slice Whitehead doubles of
non-homotopically trivial$^+$ links. We discuss further related questions in
Section~\ref{section-questions}.

\subsection*{Acknowledgements}

JCC was partly supported by National Research Foundation of Korea (NRF) grant No.\ 2019R1A3B2067839.  Part of this work was done
when JCC was visiting the Max Planck Institute for Mathematics in Bonn.  MHK
thanks Durham University and Anthony Conway for their hospitality during visits
to Durham.  MHK was partly supported by the POSCO TJ Park Science Fellowship.

\section{Seifert matrices of boundary links}
\label{section:link-module-seifert-matrix}

In this section we present certain facts on Seifert matrices of boundary links
and $S$-equivalence, which may be known to the experts. We could not find these
facts in the literature, so we provide the proofs. The main goal of the section
is to prove the following proposition.

\begin{proposition}\label{proposition:s-reducible-seifert-matrix}
  A boundary link $L$ is good if and only if there is a boundary link Seifert
  surface for $L$ such that the associated Seifert matrix can be changed to a
  null matrix by $S$-reductions only.
\end{proposition}

We will give the definitions of the terms used in
Proposition~\ref{proposition:s-reducible-seifert-matrix} in
Section~\ref{subsection:seifert-matrix-s-equivalence} below.
The following appeared as a proposition in the introduction.

\begin{corollary}
  \label{corollary:good-bd-link-iff-good-basis-exists}
  A boundary link is good if and only if it has a boundary link Seifert surface
  admitting a good basis.
\end{corollary}

We prove the corollary assuming
Proposition~\ref{proposition:s-reducible-seifert-matrix} and
Theorem~\ref{theorem:link-module-S-equivalence} below, which states that a
boundary link $L$ is good if and only if some (and therefore every) Seifert
matrix for $L$ is $S$-equivalent to the null matrix.

\begin{proof}
  Let $L$ be a good boundary link and let $V=\bigsqcup V_i$ be a boundary link
  Seifert surface for $L$ as given by
  Proposition~\ref{proposition:s-reducible-seifert-matrix}, where $V_i$ is a
  Seifert surface for the $i$th component.   We may and shall assume that each
  $V_i$ is connected.  Fix a basis for $H_1(V_i;\Z)$, for each $i$, such that
  the corresponding Seifert matrix $A$ can be changed to a null matrix by
  $S$-reductions.  The basis is symplectic with respect to the intersection
  form, since $A-A^T$, which represents the intersection form by
  Definition~\ref{definition:seifert-matrix} and the subsequent paragraph, is a
  diagonal sum of matrices of the form $\pm\sbmatrix{0 & 1 \\ -1 & 0}$, by
  Definition~\ref{definition:S-equivalence}~(ii) and~(iii) below.  Since every
  symplectic automorphism of $H_1(V_i;\Z)$ is realised by a homeomorphism
  of~$V_i$ (e.g., see~\cite[Theorem~6.4]{Farb-Margalit:2012-1}), this basis can
  be represented by simple closed curves whose geometric intersections match
  their algebraic intersection numbers.  By definition, these curves form a good
  basis.

  For the converse, suppose that a boundary link $L$ has a Seifert surface
  admitting a good basis.  Then $L$ has a Seifert matrix which is $S$-equivalent
  to a null matrix, so $L$ is good by
  Theorem~\ref{theorem:link-module-S-equivalence}.
\end{proof}

\subsection{Seifert matrices and $S$-equivalence}
\label{subsection:seifert-matrix-s-equivalence}

We begin with the basic definitions.   We always assume that an $m$-component
boundary link $L=L_1\sqcup \cdots\sqcup L_m$ is endowed with a choice of
epimorphism $\phi \colon \pi_1(S^3\sm L)\to F$ onto a free group $F$ of
rank~$m$ that sends the $i$th meridian to the $i$th free generator.  This is
equivalent to assuming, via the Pontryagin-Thom construction, that~$L$ comes
with a choice of $m$ disjoint Seifert surfaces $V_i$  with $\partial V_i = L_i$.
A boundary link with such a choice is often called an \emph{$F$-link}.

The standard definition of a Seifert matrix of a knot generalises to
the case of boundary links, as follows.

\begin{definition}
  \label{definition:seifert-matrix}
  Define a bilinear pairing $H_1(V_i;\Z)\times H_1(V_j;\Z) \to \Z$ by $([x],[y])
  \mapsto \lk(x^+,y)$ for 1-cycles $x$ and $y$, where $x^+$ denotes a push-off
  of $x$ along the positive normal direction of~$V_i$.  Choosing a basis of
  $H_1(V_i;\Z)$ for each $i$ determines a matrix~$A_{ij}$. The block matrix
  $A=[A_{ij}]_{1\le i,j\le m}$ comprising $m^2$ submatrices $A_{ij}$ is
  called a \emph{Seifert matrix} for~$L$.
\end{definition}

We remark that $A_{ii}-A_{ii}^T$ represents the intersection form of $V_i$ and
consequently is unimodular by Poincar\'{e}-Lefschetz duality on $V_i$ and the
isomorphism $H_1(V_i;\Z) \cong H_1(V_i,\partial V_i;\Z)$.  Also,
$A_{ij}=A_{ji}^T$ for $i\ne j$ by symmetry of the linking number.

The notion of $S$-equivalence for Seifert matrices of knots generalises to the
case of boundary links~\cite{Liang:1977-1}.  The next two definitions are
algebraic.

\begin{definition}
  We say that a square integral matrix $A$ is an \emph{$m$-component boundary
  link Seifert matrix} if it has a decomposition into $m^2$ blocks $A_{ij}$
  ($i,j =1,\dots,m$) with $A_{ii}$ a square matrix for each $i$, such that
  $\det(A_{ii} - A_{ii}^T) = \pm 1$ and $A_{ij} = A_{ji}^T$ for $i \neq j$.
\end{definition}

\begin{definition}
  \label{definition:S-equivalence}
  Suppose $A=[A_{ij}]$ and $B=[B_{ij}]$ are $n$-component boundary link Seifert
  matrices.
  \begin{enumerate}[label=(\roman*)]
  \item We say that $A$ and $B$ are \emph{congruent} if there are square
    unimodular matrices $P_1,\ldots, P_m$ such that $P_i^T A_{ij} P_j = B_{ij}$
    for all $i, j$. That is, for
    \[
    P=\begin{bmatrix} P_1 \\ & \ddots \\ && P_m\end{bmatrix},
    \]
    the block matrix product $P^TAP$ equals~$B$.
  \item We say that a square integral matrix $B$ is an \emph{$S$-enlargement} of
    $A$ if $B$ is of the form
    \[
    \def\V{\vphantom{\Bigg|}} \def\v{\vphantom{\Big|}} \def\s{$\;\;$}
    B=\left[\begin{array}{c|c|c@{\s}c@{\s}c|c|c}
      \ddots & \vdots & \multicolumn{3}{c|}{\vdots} & \vdots & \rddots \\ \hline
      \cdots\V & A_{k-1,k-1} & 0 & x_{k-1}^T & A_{k-1,k} & A_{k-1,k+1} & \cdots \\ \hline
      \v & 0 & 0         & \varepsilon & 0     & 0          &  \\
      \cdots & x_{k-1} & \varepsilon' & 0        & x_k   & x_{k+1}     & \cdots \\
      \v & A_{k,k-1} & 0         & x_k^T    & A_{kk} & A_{k,k+1}   & \\ \hline
      \cdots\V & A_{k+1,k-1} & 0        & x_{k+1}^T & A_{k+1,k}  & A_{k+1,k+1} & \cdots \\ \hline
      \rddots & \vdots & \multicolumn{3}{c|}{\vdots} & \vdots & \ddots \\
    \end{array}\right]
    \]
    for some~$k$, where $(\varepsilon,\varepsilon')=(1,0)$ or $(0,1)$ and $x_i$
    designates an arbitrary row vector of the appropriate dimension for
    $i=1,\ldots,m$.  That is, $B_{ij}$ is equal to
    \[
    \begin{array}{cl@{\qquad}cl}
      \begin{bmatrix}
        0 & \varepsilon' & 0 \\
        \varepsilon & 0 & x_k \\
        0 & x_k^T & A_{kk}
      \end{bmatrix}
      & \text{for } i=k, \; j=k, &
      \begin{bmatrix}
        0 \\ x_j \\ A_{kj}
      \end{bmatrix}
      & \text{for } i=k,\; j\ne k, \\[2em]
      \begin{bmatrix}
        0 & x_i^T & A_{ik}
      \end{bmatrix}
      & \text{for } i\ne k,\; j=k, &
      A_{ij} & \text{for } i\ne k,\; j\ne k.
    \end{array}
    \]
  \item We say that $B$ is an \emph{$S$-reduction} of $A$ if $A$ is an
    $S$-enlargement of~$B$.
  \item Finally, we say that $A$ and $B$ are \emph{$S$-equivalent} if
    one can be obtained from another by a finite sequence of congruences,
    $S$-enlargements, and $S$-reductions.
  \end{enumerate}
\end{definition}

\begin{proposition}[{\cite[Theorem~1]{Liang:1977-1}}]
Any two boundary link Seifert matrices of the same boundary link are
$S$-equivalent.
\end{proposition}

When $B$ is an $S$-enlargement of $A$, we write $A
\nearrow B$, or $B \searrow A$.  When $A$ and $B$ are congruent, we
write $A\sim B$.  The following lemma tells us that for any sequence
of the moves described in the definition of $S$-equivalence, a
local minimum of the size of a Seifert matrix can be replaced by a
local maximum.

\begin{lemma}
  \label{lemma:replace-min-by-max}
  If $A$ and $B$ are Seifert matrices such that $A\searrow C \sim C' \nearrow B$
  for some $C$, $C'$, then there are Seifert matrices $D$, $D'$ such that
  $A\nearrow D \sim D' \searrow B$.
\end{lemma}

\begin{proof}
  We start with a proof in the two component case, that is $m=2$.  We
  consider the case that the reduction and enlargement happen to the same block,
  since the other cases are straightforward.  Let $C=[C_{ij}]$ and $C'=P^T C P$
  for $P=\diag(P_1, P_2)$.  Then
  \[{
  \newbox\temp \setbox\temp=\hbox{$\Big|$}
  \def\vd{\vrule depth\dp\temp width 0ex}
  \def\vh{\vrule height\ht\temp width 0ex}
  A = \left[\begin{array}{ccc|c}
      0 & \varepsilon & 0 & 0 \\
      \varepsilon' & 0 & x_1 & x_2 \\
      0 & x_1^T & C_{11} & C_{12} \vd \\ \hline
      0 & x_2^2 & C_{21} & C_{22} \vh
    \end{array}\right],\;\;
  B = \left[\begin{array}{ccc|c}
      0 & \delta & 0 & 0 \\
      \delta' & 0 & y_1 & y_2 \\
      0 & y_1^T & P_1^T C_{11} P_1 & P_1^T C_{12} P_2 \vd \\
      \hline
      0 & y_2^2 & P_2^T C_{21} P_1 & P_2^T C_{22} P_2 \vh
    \end{array}\right]
  }\]
  for some $x_i$, $y_i$, $\varepsilon$, $\varepsilon'$, $\delta$, and
  $\delta'$.  Now we have $A\nearrow D \sim Q^T D Q \searrow B$ for
  \[{
  \newbox\temp \setbox\temp=\hbox{$\Big|$}
  \def\vd{\vrule depth\dp\temp width 0ex}
  \def\vh{\vrule height\ht\temp width 0ex}
  D = \left[\begin{array}{ccccc|c}
      0 & \delta & 0 & 0 & 0 & 0 \\
      \delta' & 0 & 0 & 0 & y_1P_1^{-1} & y_2P_2^{-1} \\
      0 & 0 & 0 & \varepsilon & 0 & 0 \\
      0 & 0 & \varepsilon' & 0 & x_1 & x_2 \\
      0 & (P_1^{-1})^T y_1^T & 0 & x_1^T & C_{11} & C_{12} \vd \\ \hline
      0 & (P_2^{-1})^T y_2^T & 0 & x_2^T & C_{21} & C_{22} \vh
    \end{array}\right]} \text{ and}\]
    %\;\;
\[{\newbox\temp \setbox\temp=\hbox{$\Big|$}
  \def\vd{\vrule depth\dp\temp width 0ex}
  \def\vh{\vrule height\ht\temp width 0ex}
    Q = \left[\begin{array}{ccccc|c}
      1 & & & & & \\
      & 1 & & & & \\
      & & 1 & & & \\
      & & & 1 & & \\
      & & & & P_1 & \vd \\ \hline
      & & & & & P_2 \vh \\
    \end{array}\right]
  }
  \]
  The general case is verified analogously, since as mentioned above the
  only interesting case is that the reduction and the enlargement happen to one
  block.  Without loss of generality we may assume it is the top left block.
  The matrices $D$ and $Q$ above extend in a natural way.
\end{proof}

Repeatedly apply Lemma~\ref{lemma:replace-min-by-max} to remove all local minima
of the size of $A$, so that there is at most one maximum.  We record the outcome
of this procedure.

\begin{corollary}
  \label{cor:S-equivalence-with-one-maximum}
  Suppose $A$ and $B$ are $S$-equivalent.  Then $A$ can be changed to $B$
  by a sequence of congruences, $S$-enlargements, and $S$-reductions
  in which no $S$-enlargement appears after any $S$-reduction.
\end{corollary}

Here is another elementary observation that enables us to rearrange the order of
operations applied to a Seifert matrix.

\begin{lemma}
  \label{lemma:s-reduction-congruence-switch}
  If $B$ results from $A$ after $S$-reductions and congruences, applied
  in any order, then $B$ can be obtained from $A$ by applying a single
  congruence first and then applying $S$-reductions only.
\end{lemma}

\begin{proof}
  It suffices to consider the case of $A \searrow P^TBP \sim B$.  Once again,
  for brevity, we will consider the two component case only, and in addition
  just like in proof of the preceding lemma we assume without loss of generality
  that the $S$-reduction acts on the first block of the Seifert matrix. The same
  argument applies to the general case of more than two components, with the
  natural extension of the matrices. Write $B=[B_{ij}]$ and $P=\diag(P_1, P_2)$.
  We have
  \[{
  \newbox\temp \setbox\temp=\hbox{$\Big|$}
  \def\vd{\vrule depth\dp\temp width 0ex}
  \def\vh{\vrule height\ht\temp width 0ex}
  A = \left[\begin{array}{ccc|c}
      0 & \varepsilon & 0 & 0 \\
      \varepsilon' & 0 & x_1 & x_2 \\
      0 & x_1^T & P_1^T B_{11} P_1 & P_1^T B_{12} P_2 \vd \\ \hline
      0 & x_2^2 & P_2^T B_{21} P_1 & P_2^T B_{22} P_2 \vh
    \end{array}\right]
  \searrow
  \left[\begin{array}{c|c}
      P_1^T B_{11} P_1 & P_1^T B_{12} P_2 \vd \\ \hline
      P_2^T B_{21} P_1 & P_2^T B_{22} P_2 \vh
    \end{array}\right]}
  \]
 and this equals $P^T B P \sim B$.
  Let $Q=\diag(1,1,P_1^{-1},P_2^{-1})$.  Then
  \[{
  \newbox\temp \setbox\temp=\hbox{$\Big|$}
  \def\vd{\vrule depth\dp\temp width 0ex}
  \def\vh{\vrule height\ht\temp width 0ex}
  A \sim Q^T AQ =  \left[\begin{array}{ccc|c}
      0 & \varepsilon & 0 & 0 \\
      \varepsilon' & 0 & x_1 P_1^{-1} & x_2 P_2^{-1} \\
      0 & (P_1^{-1})^T x_1^T & B_{11} &  B_{12} \vd \\ \hline
      0 & (P_2^{-1})^T x_2^2 & B_{21} &  B_{22} \vh
    \end{array}\right]
  \searrow B
  }\]
  gives a congruence followed by an $S$-reduction instead.
\end{proof}

The above observations enable us to relate two Seifert matrices by a sequence of
operations that are monotone with respect to the matrix size.

\begin{lemma}
  \label{lemma:equivalence-by-s-reduction-only}
  Suppose $L$ and $L'$ are boundary links in $S^3$ with $S$-equivalent
  Seifert matrices.  Then for any given Seifert matrix $A'$ of $L'$,
  the other link $L$ has a Seifert matrix that can be changed to $A'$ by
  a sequence of $S$-reductions.
\end{lemma}

\begin{proof}
  Let $A$ be a Seifert matrix of~$L$.  By hypothesis, $A$ is $S$-equivalent to
  $A'$.  By Corollary~\ref{cor:S-equivalence-with-one-maximum}, $A$ can be
  changed, by congruences and $S$-enlargements, to some $B$ that can then itself
  be changed to $A'$ by a sequence of congruences and $S$-reductions.  Since an
  $S$-enlargement of a Seifert matrix is realised by $1$-surgery on any given
  Seifert surface with respect to which the Seifert matrix is defined, it
  follows that there is a Seifert surface $V$ for $L$ that has $B$ as its
  Seifert matrix.  By Lemma~\ref{lemma:s-reduction-congruence-switch}, a basis
  change of the first homology transforms $B$ to a Seifert matrix for $L$
  defined on the same $V$ that can be changed to $A'$ by $S$-reductions only.
\end{proof}

\subsection{$S$-equivalence and good boundary links}
\label{subsection:seifert-matrix-good-boundary-link}

Recall that a boundary link $L$ is said to be a \emph{good boundary link} if the
given epimorphism $\phi\colon \pi_1(S^3\sm L) \to F$ has a perfect kernel, or
equivalently if $H_1(S^3\sm L;\Z F)=0$. Since the Blanchfield form on
$H_1(S^3\sm L;\Z F)$ is nondegenerate
\cite[Section~2.7]{Hillman:2012-1-second-ed}, we see that $H_1(S^3\sm L;\Z F)$
has no $\Z$-torsion.  It follows that $L$ is good if and only if $H_1(S^3\sm
L;\Q F)=0$.

%We will use the following theorem.

\begin{theorem}[\cite{Liang:1977-1}, \cite{Farber:1991-1}]
  \label{theorem:link-module-S-equivalence}
  A boundary link $L$ in $S^3$ is good if and only if some \textup{(}and thence
  every\textup{)} Seifert matrix for $L$ is $S$-equivalent to the null matrix.
\end{theorem}

In the following proof, we combine results on high dimensional
boundary links due to~\cite{Liang:1977-1, Ko:1987-1, Levine:1966-1,
Sheiham:2006-1, Farber:1991-1}.

\begin{proof}
  Let $A$ be a Seifert matrix for~$L$.  Every Seifert matrix can be realised by
 a simple boundary link in an \emph{arbitrary} odd dimension (see Theorem~2
 of~\cite{Liang:1977-1}, Theorem~3.4 of~\cite{Ko:1987-1}). Thus there is a
 simple boundary link $L'$ of codimension two in $S^{4k-1}$ with $k>1$ that has
 the same Seifert matrix~$A$.  Here a boundary link in $S^{4k-1}$ is said to be
 \emph{simple} if the components have disjoint and $(2k-2)$-connected Seifert
 surfaces.

  By~\cite{Levine:1966-1,Sheiham:2006-1}, the module $H_{2k-1}(S^{4k-1}\sm L';
  \Q F)$ has a presentation determined by the Seifert matrix.  It follows
  that the modules $H_{2k-1}(S^{4k-1}\sm L';\Q F)$ and $H_1(S^{3}\sm L;\Q
  F)$ are isomorphic.

  By \cite[Theorem~3]{Liang:1977-1}, a high dimensional simple
  boundary link is a trivial link if and only if it has a Seifert
  matrix $S$-equivalent to a null matrix.  Also, by
  \cite[Theorem~6.7]{Farber:1991-1}, $L'$ is a trivial link if and
  only if $H_{2k-1}(S^{4k-1}\sm L';\Q F)=0$.  It
  follows that $H_1(S^{3}\sm L;\Q F)=0$ if and only if the Seifert
  matrix $A$ is $S$-equivalent to the null matrix.
\end{proof}

\begin{proof}[Proof of Proposition~\ref{proposition:s-reducible-seifert-matrix}]
  Combine Theorem~\ref{theorem:link-module-S-equivalence} with
  Lemma~\ref{lemma:equivalence-by-s-reduction-only}, applied with $L'$ the
  unlink and with discs as Seifert surfaces, so $A'$ is the null matrix.   This
  shows that there is a Seifert surface $V=\bigsqcup V_i$ for $L$ such that the
  associated Seifert matrix can be changed to a null matrix by $S$-reductions
  only.
\end{proof}

\section{Homotopically trivial$^+$ pairs}

In this section, we generalise the notion of homotopically trivial$^+$ links,
defined in~\cite[Section~3]{Freedman-Teichner:1995-2}, to pairs consisting of a
link and a sublink. Recall that an $m$-component link $J$ in $S^3$ is
\emph{homotopically trivial} if there is a homotopy of $J$ to the $m$-component
unlink in $S^3$ through immersions of $m$ circles with disjoint images.

\begin{definition}
  \label{definition:htplus-pair}
  For a sublink $K$ of a link $J$ in $S^3$, we say that the pair $(J,K)$ is
  \emph{homotopically trivial\/$^+$} if $K\cup J_i^+$ is homotopically trivial
  for every component $J_i$ of $J$, where $J_i^+$ is a zero linking parallel
  copy of~$J_i$.
\end{definition}

\begin{remark}
  \phantomsection\label{remark:htplus-for-pair}
  \leavevmode\Nopagebreak
  \begin{enumerate}

    \item\label{item:htplus-good-basis-vs-pair} By inspection of the definition,
    a good basis $\{a_i,b_i\}$ on a boundary link Seifert surface $V$ is
    homotopically trivial$^+$ if and only if the pair $(\bigsqcup
    (a_i^{\mathstrut} \cup b_i'),\bigsqcup b_i')$ is homotopically trivial$^+$,
    where $b_i'$ is a normal translate of $b_i \subset V$ satisfying
    $\lk(a_i^{\mathstrut},b_i')=0$.

    \item\label{item:htplus-pair-vs-link}  The pair $(J,J)$ is homotopically
    trivial$^+$ if and only if $J$ is homotopically trivial$^+$ in the sense
    of \cite{Freedman-Teichner:1995-2}. More generally,
    suppose that $J$ is a \emph{ramification} of a link $K$, that is $J$
    consists of $K$ together with zero linking parallel copies of a (not
    necessarily proper) sublink of~$K$. Then $(J,K)$ is homotopically
    trivial$^+$ if and only if $K$ is homotopically trivial$^+$.

  \end{enumerate}
\end{remark}

The next lemma extends \cite[Lemma~3.2]{Freedman-Teichner:1995-2} to
homotopically trivial$^+$ link pairs.

\begin{lemma}
  \label{lemma:homotopy-trivial+}
  Let $(J,K)$ be a homotopically trivial\/$^+$ pair. Then there are two
  collections of immersed discs $\{\Delta_i\}$ and $\{\Delta_k^+\}$ in $D^4$
  satisfying that $\d \Delta_i=K_i$, $\d \Delta_k^+=J_k^+$, $\Delta_i\cap
  \Delta_k=\emptyset$ for all $i\neq k$ and $\Delta_i^{\vphantom{+}}\cap
  \Delta_k^+=\emptyset$ for all $i$ and~$k$.
\end{lemma}

We will give two proofs of Lemma~\ref{lemma:homotopy-trivial+}.  The first is
essentially identical to the proof
of~\cite[Lemma~3.2]{Freedman-Teichner:1995-2}, which uses the Milnor group
introduced in~\cite{Milnor:1954-1}.  The second proof gives
an elementary geometric construction of the desired discs without using algebra.
We believe that this is a useful addition to the literature.  It uses a homotopy
reversing method that appeared in~\cite[Section~3.4]{Krushkal-15}.

\begin{proof}[Proof using the Milnor group]

  Let $m$ be the number of components of~$K$.  Choose disjointly immersed discs
  $\Delta_i$ in $D^4$ bounded by~$K_i$ for $i=1,\ldots,m$, using that $K$ is
  homotopically trivial.  Let $\pi=\pi_1(S^3\sm \bigsqcup K_i)$, $G=\pi_1(D^4\sm
  \bigsqcup \Delta_i)$ and $F$ be the free group on $x_1,\ldots,x_m$.  Let $F\to
  \pi\to G$ be the map sending $x_i$ to a meridian of~$K_i$, and denote the
  image of $x_i$ in $\pi$ and $G$ by the same symbol~$x_i$. For $\Gamma = F$,
  $\pi$ and $G$, the Milnor group $M\Gamma$ is defined to be the quotient of
  $\Gamma$ modulo relators $[x_i^g,x_i^{\vphantom{g}}]$, $i=1,\ldots,m$, $g\in
  \Gamma$~\cite{Milnor:1954-1}. The group $M\Gamma$ is finitely generated and
  nilpotent~\cite{Milnor:1954-1}, and consequently finitely presented.  For
  $\Gamma=G$, since a relator $[x_i^g,x_i]$ is realised by a finger move
  on~$\Delta_i$, we may assume that $MG=G$ by applying finitely many finger
  moves to the discs. Moreover, the following facts are known: $M\pi \cong MF$
  since $K$ is homotopically trivial~\cite[Section~5]{Milnor:1954-1}, and $K\cup
  J_k^+$ is homotopically trivial if and only if $[J_k^+]$ is trivial
  in~$M\pi$~\cite[Section~5]{Milnor:1954-1}. Also,  $MF\cong MG$ since
  $\bigsqcup\Delta_i$ is a link
  null-homotopy~\cite[Lemma~2.6]{Freedman-Teichner:1995-2}. It follows that
  $[J_k^+]$ is trivial in $MG=G$ for all~$k$. Therefore there exist immersed
  discs $\Delta_k^+$ (which may intersect each other) in $D^4\sm \bigsqcup
  \Delta_i$ satisfying $\partial\Delta_k^+ = J_k^+$.
\end{proof}

\begin{proof}[Constructive geometric proof]

Let $n$ be the number of components of $J$ and let $m$ be the number of
components of~$K$.  Let $J^+$ be the union of pushed off parallel copies $J_k^+$
of the components of~$J$.

We work by induction on the components of $J^+$ to construct annuli in $S^3
\times I$, which will then be capped off by discs to obtain the desired immersed
discs, after embedding $S^3 \times I \subset D^4$ as a collar on the boundary.
For the induction, $j$ will range from $1$ to $n$, corresponding to the indices
of the components of~$J$.  Fix a 3-ball $B$ in $S^3$ disjoint
from~$K\cup J^+$. Suppose for the inductive hypothesis that we have constructed
a level preserving immersion of $n+m$ annuli $\{A_i\}_{i=1}^m$ and
$\{A_{k}^+\}_{k=1}^n$  in $S^3 \times I = S^3 \times [0,1]$ with $A_i\cap (S^3
\times \{0\}) = K_i$ for $i=1,\ldots,m$ and $A_k^+ \cap (S^3 \times \{0\}) =
J^+_k$ for $k=1,\ldots,n$, such that

\begin{enumerate}[label=(\roman*)]
\item\label{item:intersection-conditions-1} $A_i \cap A_{\ell}=\emptyset$ for
all $i\neq \ell$ with $1 \leq i,\,\ell \leq m$;
\item\label{item:intersection-conditions-2} $A_i\cap A_k^+=\emptyset$ for all
$i$ and $k$ with $1 \leq i \leq m$ and $1 \leq k \leq n$;
\item\label{item:some-lie-in-3-ball} $(A_1^+\cup\cdots \cup A_{j-1}^+) \cap (S^3
\times \{1\})$ is a trivial link contained in $B\times\{1\}$, while
$A_1,\ldots,A_m$ and $A_j^+,\ldots,A_n^+$ are disjoint from $B\times\{1\}$; and
\item\label{item:components-where-they-started}
$A_i\cap (S^3 \times \{1\}) = K_i$ for $i=1,\ldots,m$ and $A_k^+ \cap (S^3
\times \{1\}) = J^+_k$ for $k=j,\ldots,n$.
\end{enumerate}

We seek to extend these annuli in $S^3 \times [0,4]$ so that we have the same
property for $j$ replaced by $j+1$.   Then we will rescale the interval $[0,4]$
to $[0,1]=I$ to achieve the inductive step.  In the base case that $j=1$, take
the annuli to be a product concordance.

In what follows, for brevity, we denote extensions of $A_i$ and $A_k^+$ by $A_i$
and~$A_k^+$.  For $k=1,\ldots,j-1$, extend $A_k^+$ by attaching the product
$J_k^+ \times[1,2]$, which lies in $B\times[1,2]$.  Use the hypothesis that
$K\cup J_j^+$ is homotopically trivial to find disjointly immersed level
preserving annuli extending $A_1,\dots,A_m$ and $A_j^+$ in $S^3 \times [1,2]$,
ending in a trivial link in $S^3 \times \{2\}$.  We may and shall assume that
these annuli are disjoint from $B\times[1,2]$, since a link null-homotopy for
$K\cup J_j^+$ in $S^3$ can be arranged to avoid~$B$.  Furthermore, since link
homotopy is generated by crossing changes and ambient isotopy, and since
crossing changes can be assumed to be performed along arcs disjoint from
$J_{j+1}^+,\ldots,J_n^+$, the annuli $A_{j+1}^+,\ldots,A_n^+$ extend to
disjointly embedded annuli in $S^3\times[1,2]$, ending in some link in $S^3
\times \{2\}$.  The annuli $A_{j+1}^+,\ldots,A_n^+$ do not meet $B\times[1,2]$
and the above (extensions of)
$A_1^{\vphantom{+}},\ldots,A_m^{\vphantom{+}},A_j^+$.

Next, we proceed to extend the annuli in $S^3\times[2,3]$.  Choose an isotopy,
in $S^3\times\{2\}$, of the end of $A_j^+$ in $S^3 \times \{2\}$ into the
ball~$B$.  We may assume that the isotopy does not meet $A_1,\ldots,A_m$ and
$A_1^+,\ldots,A_{j-1}^+$ in $S^3\times \{2\}$.  Use the level preserving
embedded annulus corresponding to this isotopy to extend $A_j^+$ in
$S^3\times[2,3]$. Extend $A_i^{\vphantom{+}}$ and $A_k^+$, except for $k=j$,
into $S^3 \times [2,3]$ by attaching product annuli. Then $A_j^+$ is disjoint
from the annuli $A_i$ and $A_k^+$ for $k<j$, while $A_j^+$ may intersect $A_k^+$
for $k >j$. Note that these intersections are permitted in the inductive
hypothesis and ultimately in the discs that we aim to construct.

Finally, extend the annuli $A_i^{\vphantom{+}}$, $A_k^+$ once more, into $S^3
\times [3,4]$, as follows.  For the components inside $B \times \{3\}$, that is
the boundary components of the annuli $A_1^+,\dots,A_j^+$ in $S^3 \times \{3\}$,
extend them by attaching product annuli. For the remaining components, perform
the reverse of the homotopies applied thus far.  That is, for $A$ one of the
annuli $A_1,\dots,A_m$ or $A_k^+$ with $k>j$, which lies in $S^3\times[0,3]$ for
now, extend $A$ by attaching the image of $A\cap (S^3\times[1,3])$ under the
orientation reversing diffeomorphism $S^3\times[1,3] \to S^3\times[3,4]$, $(x,t)
\mapsto (x, \frac 92 - \frac t2)$.  We have arranged that the boundary
components of the annuli $A_{j+1}^+, \dots, A_n^+$ and $A_1, \dots, A_m$ that
lie in $S^3 \times \{4\}$ coincide with $J^+_{j+1},\dots,J^+_n$ and
$K_1,\dots,K_m$, so that after rescaling $[0,4]$ to $[0,1]$, condition~\ref{item:components-where-they-started} is satisfied, with $j$
replaced by $j+1$.  It is clear that no new intersections between any pair of
the annuli $A_{j+1}^+,\ldots,A_n^+$ and $A_1,\dots,A_m$ are introduced in
$S^3\times[3,4]$, unless that pair already intersected in $S^3 \times [1,3]$.
For $A_1^+,\ldots,A_j^+$, observe that their extensions in $S^3\times[3,4]$ are
contained in $B\times[3,4]$, while other annuli are disjoint from
$B\times[1,4]$, by the previous two steps.  So, there are no new intersections
involving $A_1^+,\ldots,A_j^+$.  This shows that
conditions~\ref{item:intersection-conditions-1}
and~\ref{item:intersection-conditions-2} hold for the annuli we have
constructed. Finally, we have moved the component $J_j^+$ into the ball $B$, and
so condition~\ref{item:some-lie-in-3-ball} holds with $j$ replaced by $j+1$.
This completes the proof of the inductive step.

After the induction has been run for all $j$, we obtain a collection of immersed
annuli in $S^3 \times I$ satisfying the conditions above for $j=n+1$.  Cap off
the annuli $\{A_k^+\}_{k=1}^n$ with discs inside $B \times \{1\}$.  Then cap
$S^3 \times I$ off with a 4-ball, to obtain $S^3 \times I \cup D^4 \cong D^4$.
It remains to cap off the annuli $\{A_i\}_{i=1}^m$.  The intersection of these
annuli with $S^3 \times \{1\}$ form the link~$K$. Since $K$ is homotopically
trivial, it can be capped off with disjointly immersed discs in the capping-off
$D^4$.  Let $\Delta_i$ and $\Delta_k^+$ be the discs obtained from capping off
$A_i$ and $A_k^+$ respectively. The collections of discs $\{\Delta_i\}$ and
$\{\Delta_{k}^+\}$ satisfy the desired disjointness properties.
\end{proof}

\section{Freely slicing a boundary link by finding a slice disc exterior}

Let $M_L$ denote the closed 3-manifold obtained by performing zero framed
surgery on the link $L\subset S^3$. A boundary link $L$ in $S^3$ is \emph{freely
slice} if $L$ bounds a collection of disjoint, topologically embedded,
locally flat discs $D$ in $D^4$ such that $\pi_1(D^4\sm D)$ is a free group
generated by the meridians of~$L$.  We recall the following standard
proposition.

\begin{proposition}\label{proposition:slice-disc}
  An $m$-component boundary link $L$ is freely slice if and only if there is a
  compact topological $4$-manifold $X$ with $\d X= M_L$ such that $\pi_1(X)$ is
  the free group on the $m$ meridians and $H_2(X;\Z)=0$.
\end{proposition}

We remark that such a 4-manifold $X$ is homotopy equivalent to~$\bigvee^m
S^1$ by~\cite[Proposition~11.6C(1)]{Freedman-Quinn:1990-1}.

\begin{proof}[Proof of Proposition~\ref{proposition:slice-disc}]

  If $L$ is freely slice, the exterior of the slicing discs has the properties
  required of~$X$ (we are using that locally flat embedded discs have normal
  bundles, by~\cite[Section~9.3]{Freedman-Quinn:1990-1}, in order to consider
  the exterior). For the converse, suppose there is a 4-manifold $X$ satisfying
  the stated properties. Attach $m$ $2$-handles $D^2 \times D^2$ with core
  $D^2\times\{0\}$ to $X$, with attaching circles $\partial D^2 \times \{0\}$
  glued along the meridians of $L$, with zero framings induced from~$S^3$. The
  resulting 4-manifold $Y$ is contractible, and its boundary $\d Y$ is
  homeomorphic to~$S^3$. By Freedman \cite{Freedman:1982-1}, $Y$ is homeomorphic
  to~$D^4$. The belt spheres $\{0\} \times \partial D^2$ of the 2-handles are
  sent to $L\subset S^3$ under the homeomorphism $\d Y\to S^3$, and the cocores
  $\{0\} \times D^2$  of the 2-handles provide slice discs for~$L$. Then $X$ is
  the exterior of the slice discs and $\pi_1(X)$ is free, so $L$ is freely
  slice.
\end{proof}

We will prove Theorem~\@\ref{theorem:main} by constructing a topological
$4$-manifold $X$ satisfying the criteria of
Proposition~\ref{proposition:slice-disc}.  In order to find such an~$X$, we will
need to apply a theorem of Freedman and Quinn from
~\cite[Chapter~6]{Freedman-Quinn:1990-1}. Let us introduce the terminology
necessary to state this theorem.

A \emph{model transverse pair} is defined to be two copies of $S^2\times D^2$
plumbed together once.  In other words, a regular neighbourhood of $(S^2 \times
\{p\}) \cup (\{p\} \times S^2) \cong S^2\vee S^2$ in $S^2\times S^2$, where $p
\in S^2$.  Take an arbitrary number of model transverse pairs
$N_1,\ldots,N_\ell$, perform some number of further plumbings, which may or may
not be self plumbings, and embed the result into a (topological) 4-manifold~$M$.
A topological embedding is simply a map that is a homeomorphism onto its image.
Call the resulting map $f\colon \bigsqcup_i N_i\to M$ an \emph{immersed union of
transverse pairs}. We say that $f$ has \emph{algebraically trivial
intersections} if the further plumbings introduced by $f$ can be arranged in
pairs with Whitney discs immersed in~$M$ (the Whitney discs are a priori
permitted to intersect the image of the $N_i$).  In addition, we say that $f$ is
\emph{$\pi_1$-null} if the inclusion induced map $\pi_1\big(f\big(\bigsqcup_i
N_i\big)\big)\to \pi_1(M)$ is trivial for every choice of basepoint
in~$\bigsqcup_i N_i$.

A $5$-dimensional \emph{$s$-cobordism} rel~$\d$ is a triple $(Z;\d_0 Z,\d_1 Z)$,
where $Z$ is a compact (topological) $5$-manifold whose boundary is split into
three pieces $\d Z \cong \d_0 Z\cup \d_1 Z\cup (P\times [0,1])$ such that $\d_i
Z\cap (P\times [0,1]) = \partial(\partial_i Z) = P\times\{i\}$ for $i=0,1$, and
each inclusion $\d_i Z\hookrightarrow Z$ is a simple homotopy equivalence.  We
say that $\d_0 Z$ and $\d_1 Z$ are \emph{$s$-cobordant rel~$\partial$}.

The next theorem says that, under a strong $\pi_1$-null hypothesis, we may
perform \emph{surgery up to $s$-cobordism}.

\begin{theorem}[{\cite[Theorem~6.1]{Freedman-Quinn:1990-1}}]
  \label{theorem:s-cob-to-embedding}
  Let $M$ be a compact connected topological $4$-manifold with
  \textup{(}possibly empty\textup{)} boundary. Let $f\colon \bigsqcup_i N_i\to
  M$ be an immersed union of transverse pairs. If $f$ is $\pi_1$-null and has
  algebraically trivial intersections, then $f$ is $s$-cobordant to an
  embedding.  That is, there is an $s$-cobordism $(Z;M,M')$ rel~$\d$ and a map
  $F\colon \bigsqcup_i N_i \times I \to Z$ such that $F|_{\bigsqcup_i N_i\times
  \{0\}}=f$ and $F|_{\bigsqcup_i N_i \times \{1\}}$ is an embedding of
  $\bigsqcup_i N_i$ into~$M'$.
\end{theorem}

In the next section we will prove the following result, which forms the core of
the proof of Theorem~\@\ref{theorem:main}\@.

\begin{proposition}\label{proposition:construction-of-W}
  Let $L$ be an $m$-component boundary link with a boundary link Seifert surface
  admitting a homotopically trivial\/$^+$ good basis. Then there exists a
  compact smooth $4$-manifold $W$ with $\partial W = M_L$, $\pi_1(W)$ free on
  the $m$ meridians of~$L$, and $H_2(W;\Z) \cong \Z^{2g}$ for some $g\ge 0$.
  Moreover, there is a basis for $H_2(W;\Z)$ represented by a $\pi_1$-null
  immersed union of transverse pairs with algebraically trivial intersections.
\end{proposition}

Assuming the proposition, we give the proof of our  main theorem, whose
statement we recall for the convenience of the reader. \emph{Suppose that $L$ is
a good boundary link that has a Seifert surface admitting a homotopically
trivial\/$^+$ good basis. Then $L$ is freely slice.}

\begin{proof}
  [Proof of Theorem~\@\ref{theorem:main} assuming
  Proposition~\ref{proposition:construction-of-W}]

  Apply Theorem~\ref{theorem:s-cob-to-embedding} to $W$, to find a $4$-manifold
  $W'$ $s$-cobordant rel~$\d$ to $W$, via an $s$-cobordism $(Z;W,W')$ with maps
  \[F_1,\ldots,F_{2g}\colon S^2\times [0,1]\to Z\] such that the immersed
  spheres $F_i|_{S^2\times \{0\}}\colon S^2\to W$ represent a basis for
  $H_2(W;\Z)$ and $h_i:=F_i|_{S^2\times \{1\}}\colon S^2\to W'$ for
  $i=1,\dots,2g$ are framed embeddings with geometric intersections
  \[
    |h_{i}\pitchfork h_j|= \begin{cases}
      1 & \text{if } \{i,j\}=\{2k-1,2k\} \text{ for some } k, \\
      0 & \text{otherwise.}
    \end{cases}
  \]
  Since $Z$ is an $s$-cobordism, $\{h_i(S^2)\}$ forms a basis of $H_2(W';\Z)$ as
  well. Perform surgery on the $2$-spheres $h_2(S^2), h_4(S^2),\ldots,
  h_{2g}(S^2)$ to convert $W'$ into a compact topological $4$-manifold~$W''$
  with $\d W''\cong \d W'\cong \d W \cong M_L$.  Since each of the spheres
  $h_{2k}(S^2)$ has a geometric dual, the surgeries kill $H_2$ and preserve
  $\pi_1$. Thus $\pi_1(W'')$ is free on the meridians of~$L$ and
  $H_2(W'';\Z)=0$.  By Proposition~\ref{proposition:slice-disc} with $X=W''$,
  $L$ is freely slice.
\end{proof}

\section{Construction of the 4-manifold \texorpdfstring{$W$}{W}}

Our goal is to construct a 4-manifold $W$ with the properties stated in
Proposition~\ref{proposition:construction-of-W}.  The construction in this
section will be done completely in the smooth category.  The description of $W$
given below is closely related to a cobordism used in the proof
of~\cite[Theorem~8.9]{Cochran-Orr-Teichner:1999-1}.   Let $L$ be an
$m$-component boundary link with a boundary link Seifert surface
$V=\bigsqcup_{k=1}^m V_k$ that admits a homotopically trivial$^+$ good
basis~$\{a_i,b_i\}$.  Let $g_k$ be the genus of $V_k$ and let $g=\sum_{k=1}^m
g_k$.  Take a bicollar of $V_k$ in $S^3$ and identify the bicollar with
$V_k\times[-1,1]$. The bicollar consists of one 3-dimensional 0-handle $D^0
\times D^3$ and $2g_k$ 1-handles $D^1 \times D^2$ corresponding to the curves
$a_i$, $b_i$, as shown in Figure~\ref{figure:diagram-on-surface}. For $i$ such
that $a_i\cup b_i\subset V_k$, let $\beta_i$, $\gamma_i$ and $\delta_i$ be the
curves in $V_k\times [-1,1]$ shown in Figure~\ref{figure:diagram-on-surface}.
The curve $\beta_i$ bounds a disc in~$V_k=V_k\times 0$.   The $4g-m$ curves
$a_i$, $\beta_i$, $\gamma_i$ and $\delta_i$, viewed as curves in $S^3$ under the
bicollar inclusions, form a Kirby diagram.  The dotted curves $\beta_i$
represent 1-handles, while $a_i$, $\gamma_i$ and $\delta_i$ are zero-framed
attaching circles for 2-handles. Note that the zero framing of our curves in the
standard embedding picture in Figure~\ref{figure:diagram-on-surface} corresponds
to the zero framing under the bicollar embedding, since the Seifert pairing of
$V_k$ vanishes on $(a_i,a_i)$ and~$(b_i,b_i)$.  Let $W$ be the 4-manifold
described by this Kirby diagram.

\begin{figure}[htb!]
  \includegraphics[scale=.5]{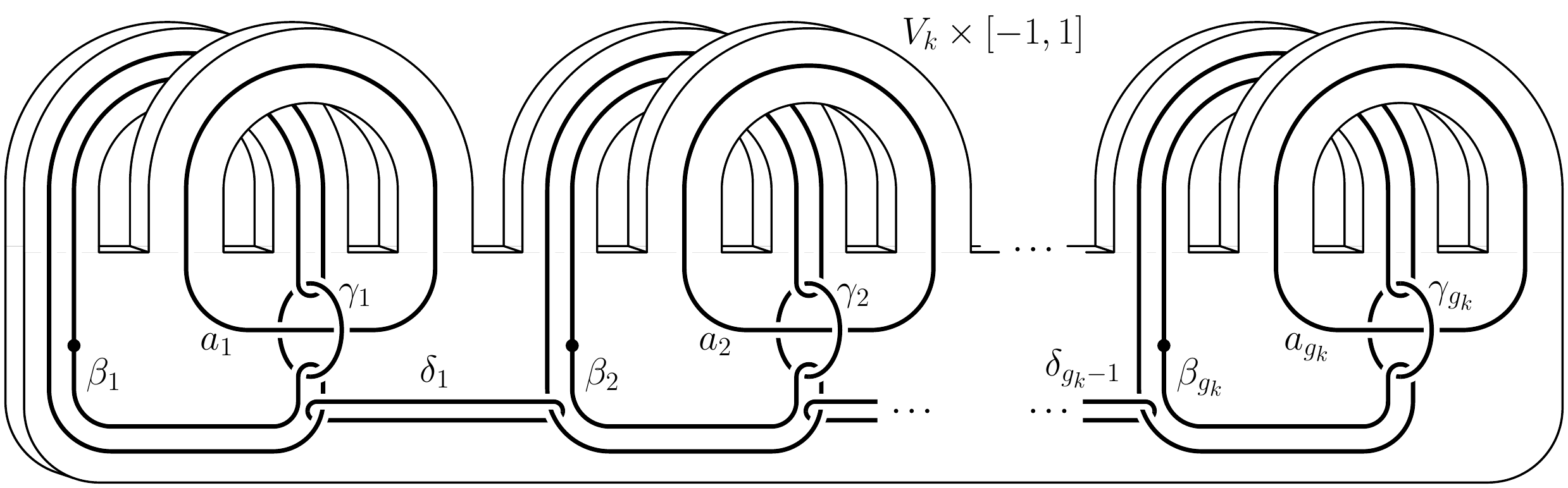}
  \caption{A Kirby diagram consisting of $a_i$, $\beta_i$, $\gamma_i$ and
  $\delta_i$ in the bicollar $V_k\times[-1,1]$.  For brevity, the indices $i$
  such that $a_i\cup b_i\subset V_k$ are assumed to be $1,2,\ldots,g_k$.  The
  curves $a_i$, $\gamma_i$ and $\delta_i$ are all zero framed.}
  \label{figure:diagram-on-surface}
\end{figure}

The boundary of $W$ is diffeomorphic to the zero surgery manifold~$M_L$, by the
arguments of~\cite[pp.~147--148]{Freedman:1993-1}
and~\cite[Lemma~8.10]{Cochran-Orr-Teichner:1999-1}. We describe the proof since
we need to keep track of meridians under $M_L\cong \partial W$.  Regard the
Kirby diagram as a surgery presentation of the 3-manifold $\partial W$
consisting of zero framed curves, ignoring the dots on~$\beta_i$.  Slide each
$\beta_i$ over $a_i$ twice, once from the inside of $a_i$ and once from the
outside, to obtain the surgery diagram in
Figure~\ref{figure:diagram-after-slide}.  We abuse notation and still denote the
resulting curves by~$\beta_i$.

\begin{figure}[htb!]
  \includegraphics[scale=.5]{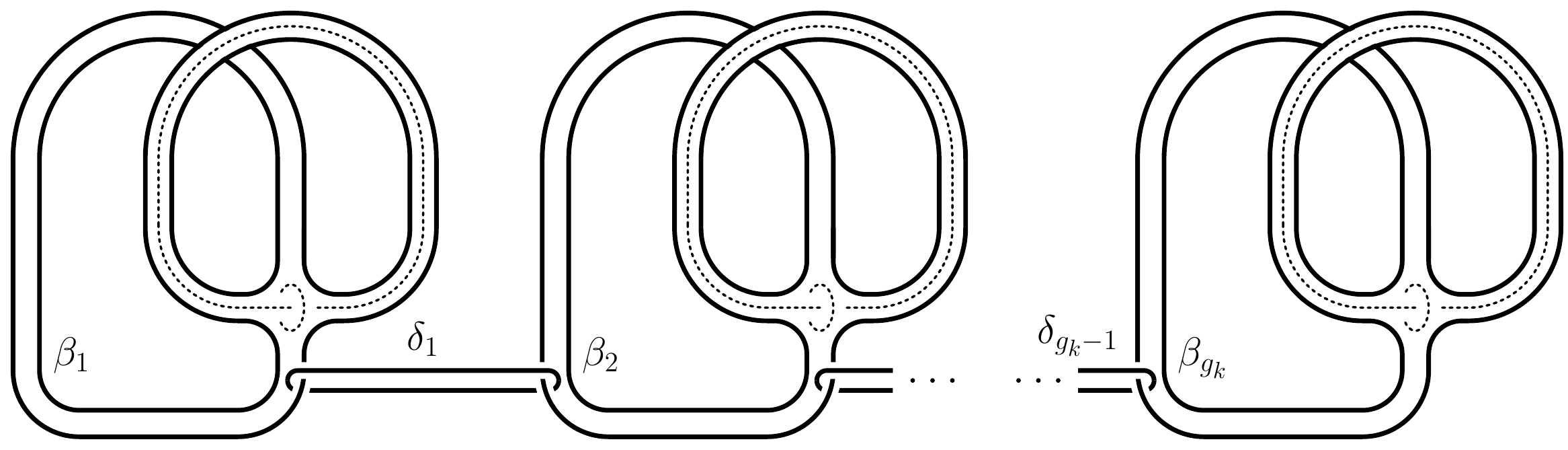}
  \caption{The surgery curves obtained after handle sliding. The curves
  $\beta_i$ and $\delta_i$ are all zero framed.}
  \label{figure:diagram-after-slide}
\end{figure}

Eliminate, in pairs, the curves $a_i$ and $\gamma_i$ shown as dashed
circles in Figure~\ref{figure:diagram-after-slide}.  Slide the new $\beta_1$,
$g_k-1$ times, toward the right hand side over the new
$\beta_2,\ldots,\beta_{g_k}$, and finally eliminate $\delta_{i}$ and
$\beta_{i+1}$ in pairs for $i=1,\ldots,g_k-1$.  Now the diagram consists of zero
framed curves isotopic to~$L$.  This shows that $\partial W \cong M_L$, as
desired.  In addition, for $\beta_i\subset V_k$, a meridian of $\beta_i$ is sent
to a meridian of the $k$th component of $L$ under this diffeomorphism.

\begin{claim}
  The fundamental group $\pi_1(W)$ is the free group on the $m$ meridians
  of~$L$.
\end{claim}

\begin{proof}
  To compute $\pi_1(W)$, begin with $\pi_1(D^4 \cup{}$1-handles$)$, which is
  free on the meridians of the curves~$\beta_i$. The 2-handles attached along
  $a_i$ and $\gamma_i$ do not give relators, since $a_i$ and $\gamma_i$ are
  null-homotopic in $D^4 \cup($1-handles$)$.  This can be seen, for instance,
  by inspecting their intersections with the disc in $V_k$ bounded
  by~$\beta_i$. The 2-handles attached along $\delta_i$ give relations
  identifying meridians of all $\beta_i$ lying on the same~$V_k$.  So we have
  $m$ generators, one for each~$V_k$.  We already verified that these generators
  are meridians of $L$ in $\partial W \cong M_L$.
\end{proof}

Next, we construct $2g$ immersed spheres
$\Sigma_1,\ldots,\Sigma_{2g}$ in~$W$. Apply Lemma~\ref{lemma:homotopy-trivial+}
to $J=\{a_i^{\mathstrut},b_i'\}_{i=1}^{g}$ and $K=\{b_i'\}_{i=1}^{g}$ to obtain
collections of discs $\{\Delta_i^+\}_{i=1}^{2g}$ and $\{\Delta_i\}_{i=1}^g$ in
$D^4$ satisfying
\[
  \begin{array}{@{}l@{\quad}l}
    \partial\Delta_{i}^+ = a_i,\, \partial\Delta_{g+i}^+ = b_i'^+ ,\,
    \partial\Delta_i = b_i' &
    \text{for }1\le i \le g,
    \\
    \Delta_i\cap \Delta_j=\emptyset &\text{for }i\ne j, \text{ and}
    \\
    \Delta_i^{\mathstrut}\cap \Delta_j^+ = \emptyset &\text{for all $i$
    and~$j$.}
  \end{array}
\]
Here $b_i'^+$ is a zero linking parallel of $b_i'$ as before. For each of these
discs, perform interior twists~\cite[Section~1.3]{Freedman-Quinn:1990-1} if
necessary to arrange that the signs of the self intersections add up to zero. It
follows that the discs induce the zero framing on the boundary.  Note that the
discs can also be assumed to be smoothly immersed and transverse to each other.
We will use the discs to build our immersed spheres.

First, for $i=1,\ldots,g$, cap off the disc $\Delta_i^+$ with the core
$D^2\times \{0\}$ of the 2-handle $D^2 \times D^2$ attached
along~$a_i=\partial\Delta_i^+$.  Let
\[
  \Sigma_{2i-1} := (D^2\times\{0\}) \cupover{a_i} \Delta_i^+ \quad(i=1,\ldots,g)
\]
be the resulting immersed spheres.  They are half of our spheres.  Since the
2-handle is attached along the zero framing, the normal bundle of
$\Sigma_{2i-1}$ is trivial, so $\Sigma_{2i-1}$ is framed.

Now we define the spheres $\Sigma_{2i}$, for $i=1,\dots,g$. Take the genus one
surface $T_i^\circ \subset W$ shown in Figure~\ref{figure:toral-generator}.  It
is bounded by $\gamma_i$ and contains~$b_i$.  Cap off $T_i^\circ$ with the core
of the 2-handle attached along $\gamma_i$, to obtain a torus in $W$, which we
call~$T_i$.

\begin{figure}[htb!]
  \includegraphics[scale=.5]{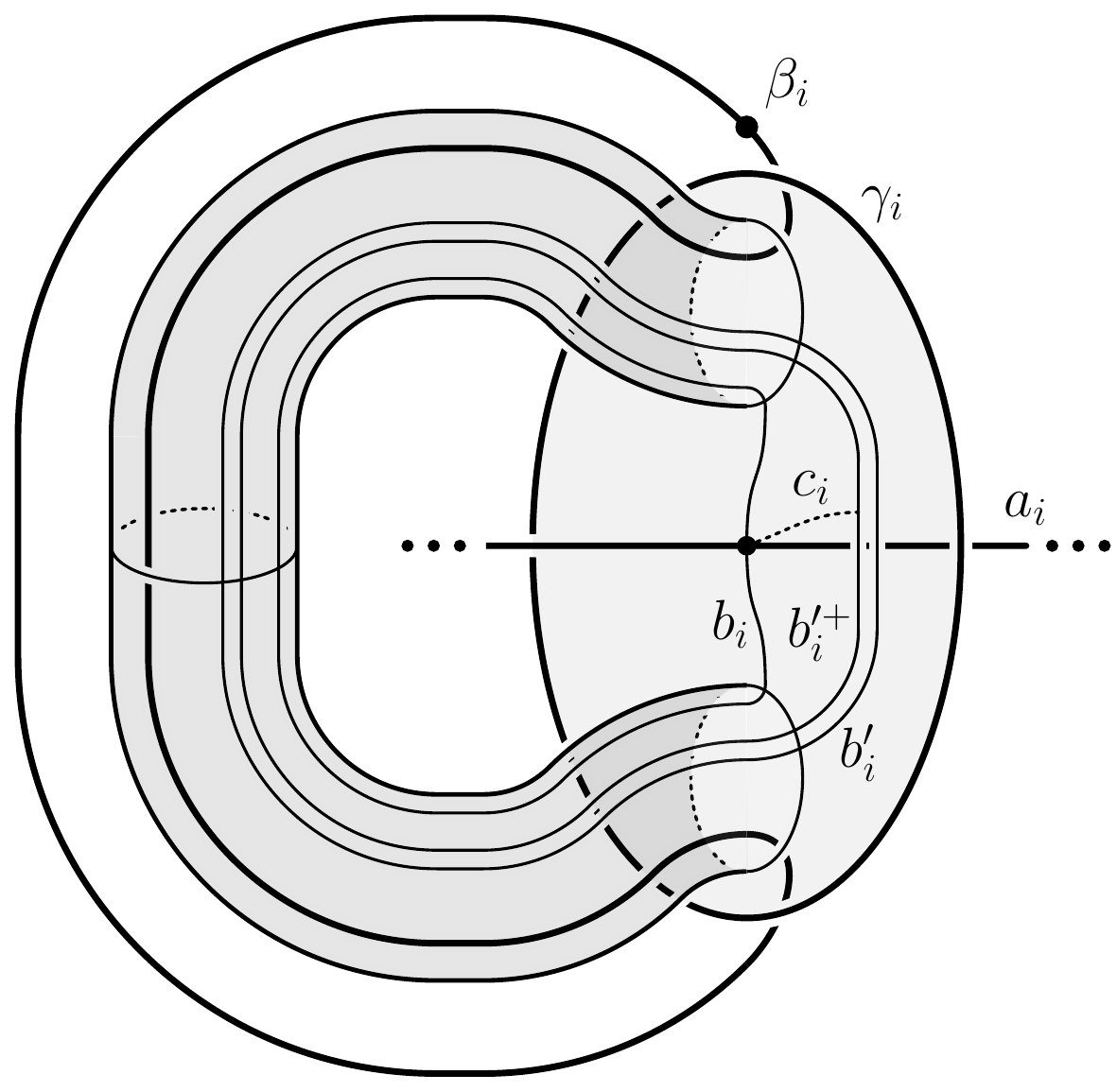}
  \caption{The genus one surface $T_i^\circ$ with $\partial T_i^\circ=\gamma_i$.
  The push-off $b_i'$ of $b_i$ is either the one shown, or a parallel on the
  opposite side of $b_i$, depending on the normal direction used to take~$b_i'$.
  The parallel $b_i'^+$ is between $b_i$ and~$b_i'$.  The dashed arc $c_i$ on
  $T_i^\circ$ joins the intersection $a_i \cap T_i^\circ = a_i\cap b_i$ to the
  parallel~$b_i'^+$, without intersecting~$b_i'$.}
  \label{figure:toral-generator}
\end{figure}

Note that the normal direction of $b_i$ in $T_i^\circ$ agrees with a normal
direction of the Seifert surface.  So we may assume that $b_i'$ is a parallel of
$b_i$ on~$T_i^\circ$, as shown in Figure~\ref{figure:toral-generator}. Also,
since the Seifert pairing vanishes on $(b_i,b_i)$, we may assume that the zero
linking parallel $b_i'^+$ is on $T_i^\circ$ too.  In addition, we assume that
$b_i'^+$ is a translate of $b_i'$ taken along the normal direction of $b_i'$
\emph{toward} $b_i$, as shown in Figure~\ref{figure:toral-generator}.  That is,
$b_i'^+$ lies between $b_i'$ and $b_i$ on~$T_i^\circ$.  (This will be crucial in
verifying $\pi_1$-nullity later.) Contract $T_i$ using the discs $\Delta_{i}$
and $\Delta_{g+i}^+$ to obtain a 2-sphere~$\Sigma_{2i}$.  More
precisely, take the annulus cobounded by $b_i'$ and $b_i'^+$ in $T_i^\circ$, and
remove it from $T_i$ to get the complementary annulus~$A_i$.  Attach the discs
$\Delta_{i}^{\mathstrut}$ and $\Delta_{g+i}^+$ to $A_i$ along the boundary, to
construct our 2-spheres
\[
  \Sigma_{2i} := \Delta_i \cupover{b_i'} A_i \cupover{b_i'^+} \Delta_{g+i}^+
  \quad(i=1,\dots,g).
\]
Since $\Delta_i^{\mathstrut}$, $\Delta_{g+i}^+$ and $A_i$ induce the zero
framing on $b_i'$ and $b_i'^+$, $\Sigma_{2i}$ is framed.

\begin{figure}[htb!]
  \centering
  \includegraphics[scale=.9]{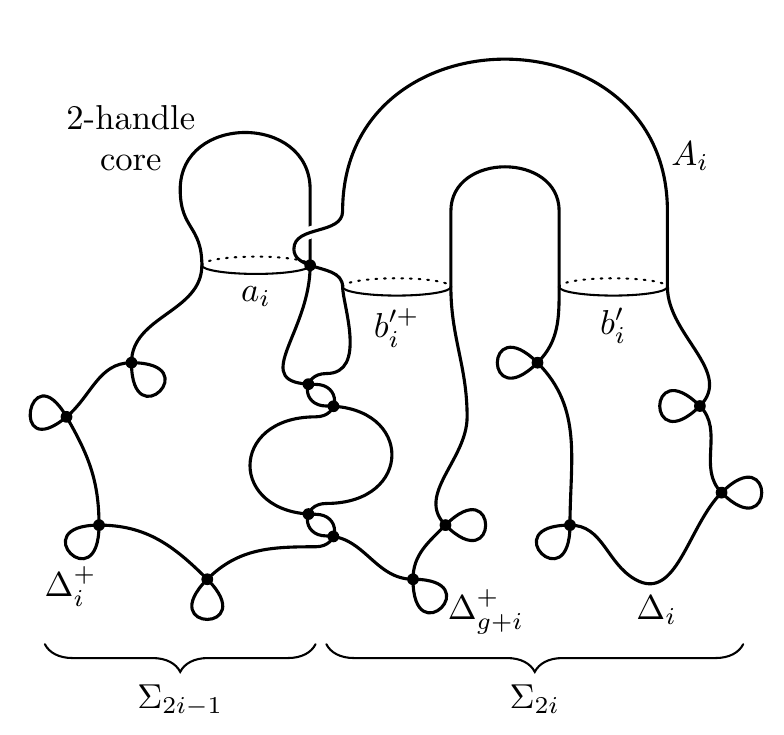}
  \caption{A schematic diagram of construction of $\Sigma_{2i-1}$
    and~$\Sigma_{2i}$.  The dot on the circle $a_i$ represents the intersection
    $a_i\cap A_i$ shown in Figure~\ref{figure:toral-generator}.  Other dots
    represent potentially existing intersections.}
  \label{figure:spherical-generators}
\end{figure}

\begin{claim}
  The homology $H_2(W;\Z)$ is the free abelian group of rank $2g$ generated by
  the classes of~$\Sigma_i$.
\end{claim}

\begin{proof}
Consider the cellular chain complex associated to the handle decomposition
of~$W$. In particular we have $C_2(W;\Z)=H_2(W,W^{(1)};\Z)$ where
$W^{(1)}=D^4\cup\text{1-handles}$.  Inspect the classes of 2-handle attaching
curves in $H_1(W^{(1)};\Z)$, in the $\pi_1$ computation above, to verify
that $H_2(W;\Z)\cong \Z^{2g}$ and that the 2-handle cores bounded by $a_i$ and
$\gamma_i$ are generators.  Note that each $\Sigma_i$ uses exactly one of these
2-handle cores, and that other building blocks $\Delta_i^{\mathstrut}$,
$\Delta_i^+$ and $T_i^\circ$ lie in~$W^{(1)}$.  It follows that $\Sigma_i$
equals the corresponding 2-handle core as a $2$-cycle in the cellular chain
complex~$C_2(W;\Z)=H_2(W,W^{(1)};\Z)$.
\end{proof}

Note that the framed spheres $\Sigma_{2i-1}$ and $\Sigma_{2i}$ have a
distinguished transverse intersection at~$a_i\cap T_i=a_i\cap b_i$.  This is
also depicted in Figure~\ref{figure:toral-generator}.  So, (a regular
neighborhood of) $\Sigma_{2i-1}\cup \Sigma_{2i}$ can be viewed as an immersed
pair of transverse spheres.

\begin{claim}
  The union $\bigcup_{i=1}^{2g} \Sigma_i$ of the $g$ pairs has
  algebraically trivial intersections.
\end{claim}

\begin{proof}
First, we investigate the intersections of the discs $\Delta_i$
and~$\Delta_i^+$. Recall that the signs of the self intersections of each disc
add up to zero, that is, the $\Z$-valued self intersection number vanishes.
Since the discs lie in $D^4$, which is simply connected, it follows that all the
self intersections can be paired up by (possibly immersed) Whitney discs.  For
$\Delta_i^+$ and $\Delta_j^+$ with $i\ne j$, their $\Z$-valued intersection
number vanishes because the linking number of $\partial\Delta_i^+$ and $\partial
\Delta_j^+$ is zero.  So, their intersections can also be paired up by Whitney
discs in~$D^4$. Moreover, recall that $\Delta_i\cap \Delta_j=\emptyset$ for
$i\ne j$  and $\Delta_i^+ \cap \Delta_j = \emptyset$ for all $i$ and~$j$.

Now, consider the spheres~$\Sigma_i$.  Since they are built using $\Delta_i$,
$\Delta_i^+$, $A_i$ and 2-handle cores, the (self and non-self) intersections of
the spheres, with the exception of the distinguished transverse intersection of
the dual pair $(\Sigma_{2i-1}, \Sigma_{2i})$ at $a_i\cap T_i = a_i \cap b_i$,
are exactly the (self and non-self) intersections of the discs $\Delta_i$ and
$\Delta_i^+$. These can all be paired up by immersed Whitney discs by the
preceding paragraph.  This proves the claim.
\end{proof}

\begin{claim}
  The union $\bigcup_{i=1}^{2g} \Sigma_i$ is $\pi_1$-null in~$W$.
\end{claim}

\begin{proof}
Recall that $\Sigma_{2i-1}$ is the union of $\Delta_i^+$ and a 2-handle core
disc.  Remove these core discs from $\bigcup_{i=1}^{2g} \Sigma_i$ to get a
2-complex~$Y$. That is, $Y=\bigcup_{i=1}^g (\Delta_i^+ \cup \Sigma_{2i})$. Since
the inclusion induces an isomorphism $\pi_1(Y)\cong \pi_1(\bigcup_i \Sigma_i)$
for any basepoint, $\bigcup_{i=1}^{2g} \Sigma_i$ is $\pi_1$-null if and only if
$Y$ is $\pi_1$-null in~$W$.  Remove from $Y$ the discs $\Delta_i$ used to
construct $\Sigma_{2i}$, for all $i$, to obtain a 2-complex~$Y_0$ given
by
\[
  Y_0 := \bigcup_{i=1}^g \Big( \Delta_i^+ \cup \big(
    A_i \cupover{b_i'^+} \Delta_{g+i}^+ \big)\Big).
\]
Since $(\operatorname{int}\Delta_i) \cap Y_0 =\emptyset$ and $\Delta_i \cap
\Delta_j=\emptyset$ for $i\ne j$, $\pi_1(Y)$ is the free product of $\pi_1(Y_0)$
and the groups $\pi_1(\Delta_i)$, by the Seifert-van Kampen theorem. Each
$\Delta_i$ is $\pi_1$-null since $\Delta_i$ is contained in~$D^4$.  Therefore,
$Y$ is $\pi_1$-null if and only if $Y_0$ is $\pi_1$-null.

Now, consider the arc $c_i$ on $A_i$ shown in
Figure~\ref{figure:toral-generator}, which joins the point $a_i\cap A_i$ to the
boundary circle~$b_i'^+$.  From Figure~\ref{figure:toral-generator}, we see that
the annulus $A_i$ strongly deformation retracts onto the 1-subcomplex $c_i\cup
b_i'^+$. It follows that $Y_0$ strongly deformation retracts onto
\[
  Y_1 = \bigcup_{i=1}^g \big( \Delta_{2i-1}^+ \cup c_i \cup \Delta_{2i}^+ \big).
\]
Observe that $c_i$ is contained in $W^{(1)}=D^4\cup{}$1-handles.  Moreover,
$c_i$ is disjoint from the 1-handles of~$W$, since $c_i$ does not intersect the
obvious 2-disc in $V_k$ bounded by the dotted circle~$\beta_i$, as seen in
Figure~\ref{figure:toral-generator} (see also
Figure~\ref{figure:diagram-on-surface}).  Here, it is crucial that the parallel
$b_i'^+$ lies between $b_i$ and~$b_i'$; if $b_i'^+$ were on the opposite side of
$b_i'$, then $c_i$ would meet (the cocore of) the 1-handle corresponding
to~$\beta_i$. It follows that the 2-complex $Y_1$ is contained in~$D^4$.
Therefore $Y_1$ is $\pi_1$-null, and consequently $Y_0$, $Y$ and
$\bigcup_{i=1}^{2g}\Sigma_i$ are all $\pi_1$-null.  This proves the claim, and
completes the proof of Proposition~\ref{proposition:construction-of-W}.
\end{proof}

\begin{remark}
  In the above $\pi_1$-nullity proof, note that there might have been a
  homotopically essential loop in~$\Sigma_{2i-1}\cup \Sigma_{2i}$, passing
  through the distinguished transverse intersection $a_i\cap A_i$ exactly once,
  from one sphere to another: start from $a_i\cap A_i$, depart into
  $\Delta_i^+\subset \Sigma_{2i-1}$ and proceed to reach an intersection with
  $\Delta_{g+i}^+ \subset \Sigma_{2i}$, and then go back to the point $a_i\cap
  A_i$ through~$\Sigma_{2i}$. See Figure~\ref{figure:spherical-generators}.
  This occurs since the discs $\Delta_i^+$ and $\Delta_{g+i}^+$ used to
  construct the dual 2-spheres $\Sigma_{2i-1}$ and $\Sigma_{2i}$ are allowed to
  meet. Our subtle choice of $b_i'^+$ from the two possibilities, together with
  the use of the arc $c_i$, is crucial in showing that such a loop is
  null-homotopic in~$W$.

  This subtlety is peculiar to our case, compared with the proof of the
  earlier result of Freedman and Teichner~\cite{Freedman-Teichner:1995-2} that a
  Whitehead double of a homotopically trivial$^+$ link is freely slice. In their
  construction, immersed discs used to construct a pair of dual 2-spheres are
  always disjoint, and thus their verification of the $\pi_1$-nullity does not
  involve this complication of our case.
\end{remark}

\begin{remark}
  It can also be shown that $\pi_2(W)=H_2(W;\Z[\pi_1(W)])$ is the free module
  of rank $2g$ over the group ring $\Z[\pi_1(W)]$ generated by our
  spheres~$\Sigma_{i}$.  Since we do not use this, we omit the details.
\end{remark}

\section{Comparison with previously known results}
\label{section:previous-results-comparison}

Here is a summary of previously known results on freely slicing good boundary
links. Recall that any Whitehead double of a pairwise linking number zero link
is a good boundary link.  For statement~\eqref{item:Freedman93} below, we need
some terminology from~\cite{Freedman:1993-1}.  A link is said to be a
\emph{$\d^2$-link} if it bounds a boundary link Seifert surface $V$ with basis
$\{\alpha_i,\beta_i\}$ for $H_1(V;\Z)$, for which the Seifert matrix is the
diagonal sum of matrices of the form $\sbmatrix{0 & \varepsilon \\ 1-\varepsilon
& 0}$ with $\varepsilon=0$ or $1$, and the homology classes $\beta_i\in
H_1(V;\Z)$ can be realised by a collection of disjoint simple closed curves
$b_i$ on $V$ that when considered as a link in its own right is itself a
boundary link~$J$. Moreover we require that $J$ bounds a boundary link Seifert
surface in $S^3$ whose interior is disjoint from~$V$.

\begin{enumerate}
  \item\label{item:Freedman82} Freedman \cite{Freedman:1982-2} showed that a
  knot is freely slice if and only if it has Alexander polynomial one, that is
  if it is a good boundary link.
  \item\label{item:Freedman85} Freedman \cite{Freedman:1985-1} showed that
  Whitehead doubles of boundary links are freely slice good boundary links.
  \item\label{item:Freedman88} Freedman \cite{Freedman:1988-2} showed that a
  Whitehead double of a 2-component link $L$ is freely slice if and only if $L$
  has linking number zero.
  \item\label{item:Freedman93} Freedman \cite{Freedman:1993-1} showed that a
  $\partial^2$-link $L$ is freely slice.
  \item\label{item:FreedmanTeichner95} Freedman and Teichner
  \cite{Freedman-Teichner:1995-2} showed that a Whitehead double of a
  homotopically trivial$^+$ link is freely slice.
\end{enumerate}

In addition to the above list of results on freely slice boundary links,
Cochran-Friedl-Teichner \cite{Cochran-Friedl-Teichner:2006-1} produced new slice
links from a satellite construction, whilst the first and third named
authors~\cite[Theorem~F]{Cha-Powell:2016-1} showed that any iterated ramified
4-fold Whitehead double is slice. Moreover, in~\cite{Cha-Powell:2016-1}, they
also introduced distorted Whitehead doubles, and showed that they are slice too.
Neither the results of \cite{Cochran-Friedl-Teichner:2006-1} nor
\cite{Cha-Powell:2016-1} address whether the slice discs are free.

Our result does not recover statement~\eqref{item:Freedman82}. Instead, for
an Alexander polynomial one knot $K$, form the $4$-manifold $W_0$ as in our
proof. Then the generators of $\pi_2(W_0)$ that we construct need not be
$\pi_1$-null. But, since $\pi_1(W_0)$ is the good group~$\Z$, it turns out that
one can use ordinary topological surgery to produce a slice disc exterior for
$K$, instead of surgery up to $s$-cobordism.  This is essentially the argument
of~\cite[Theorem~11.7B]{Freedman-Quinn:1990-1}.

In what follows, we discuss the relationship of our theorem with the
work~\eqref{item:Freedman85}--\eqref{item:FreedmanTeichner95}.  Also, with
reference to \eqref{item:Freedman85}--\eqref{item:FreedmanTeichner95}, in
Section~\ref{subsection:examples} we will provide examples of links that are
freely slice by our Theorem~\@\ref{theorem:main}, but which are neither
Whitehead doubles nor $\partial^2$-links.

\begin{lemma}
  \label{lemma:wh-and-ht+}
  Let $J$ be a link with trivial pairwise linking numbers.  Then the
  standard Seifert surface for the boundary link $\Wh(J)$ consisting of genus
  one surfaces has a good basis which is homotopically trivial\/$^+$ if and
  only if $J$ is a homotopically trivial\/$^+$ link.
\end{lemma}

\begin{proof}
  Let $J_i$ be the $i$th component of~$J$. Figure~\ref{figure:good-basis-for-wh}
  shows the standard genus one Seifert surface for the Whitehead double
  of~$J_i$.  Let $a_i$ and $b_i$ be the curves on the Seifert surface shown in
  Figure~\ref{figure:good-basis-for-wh}.  Since $J$ has trivial linking numbers,
  the Seifert matrix for the boundary link $\Wh(J)$ with respect to
  $\{a_i,b_i\}$ is the diagonal sum of $2\times 2$ matrices of the form
  $\sbmatrix{0 & \varepsilon \\ 1-\varepsilon & 0}$ with $\varepsilon=0$ or~$1$.
  It follows that $\{a_i,b_i\}$ is a good basis.

  As in the definition of a homotopically trivial$^+$ good basis, let $b_i'$ be
  a translate of $b_i$ along a normal direction to the Seifert surface such that
  $\lk(a_i,b_i')=0$, and let $K :=\bigsqcup b_i'$. Observe that $a_i$, $b_i$ and
  $b_i'$ are zero linking parallels of~$J_i$.  Thus both $K\cup a_i$ and $K\cup
  b_i$ are isotopic to $J\cup J_i^+$. It follows that the link $J$ is
  homotopically trivial$^+$ if and only if the good basis $\{a_i,b_i\}$ is
  homotopically trivial$^+$.
\end{proof}

Thus \eqref{item:FreedmanTeichner95} is a corollary of
Theorem~\ref{theorem:main}, as mentioned in the introduction.

A link $L$ is homotopically trivial$^+$ if and only if all of Milnor's
invariants with at most one repeated index are
trivial~\cite[Lemma~2.7]{Freedman-Teichner:1995-2}.  In particular, boundary
links are homotopically trivial$^+$ since all their Milnor's invariants vanish.
Also, a 2-component link with trivial linking number is homotopically
trivial$^+$, since $\bar\mu(112)=\bar\mu(122)=0$ for any such link, due to
cyclic symmetry and the shuffle relation~\cite[Theorem~6]{Milnor:1957-1}. It
follows that \eqref{item:FreedmanTeichner95} implies \eqref{item:Freedman85} and
\eqref{item:Freedman88}.  Therefore \eqref{item:Freedman85} and
\eqref{item:Freedman88} are corollaries of Theorem~\ref{theorem:main} too.

By the following lemma, Theorem~\ref{theorem:main}
implies~\eqref{item:Freedman93}.

\begin{lemma}
  \label{lemma:b^2-links-have-ht+-good-basis}
  A $\d^2$-link has a boundary link Seifert surface admitting homotopically
  trivial\/$^+$ good basis.
\end{lemma}

\begin{proof}
  Suppose that $L$ is a $\d^2$-link.  Choose a boundary link Seifert surface $V$
  and disjoint simple closed curves $b_i$ on $V$ with the properties described
  in the definition of a $\d^2$-link.  The link $\bigsqcup b_i$ bounds a
  boundary link Seifert surface $W$ that is transverse to $V$ and satisfies
  $V\cap W=\bigsqcup b_i$.  Choose disjoint simple curves $a_i$ on $V$ that are
  geometrically dual to the~$b_i$. That is, the geometric intersections are
  given by $a_i\cdot b_j=\delta_{ij}$.  Here, we allow that the homology class
  of $a_i$ may be different from the $\alpha_i$ in $H_1(V;\Z)$ given in the
  definition of a $\d^2$-link.  Let $Y$ be the surface $W$ with an open collar
  of $\partial W$ removed. Then $\partial Y$ consists of push-offs $b_i'$ of the
  $b_i$, and $Y$ is a boundary link Seifert surface for the link $K=\bigsqcup
  b_i'$.

  We claim, in general, that if $K$ is a boundary link with boundary link
  Seifert surface $Y$ and $\gamma$ is a knot disjoint from $Y$, then
  $K\cup\gamma$ is homotopically trivial.  Applying this to our case, it follows
  that  $K\cup b_i$ and $K\cup a_i$ are homotopically trivial for each~$i$. This
  shows that $\{a_i,b_i\}$ is a homotopically trivial$^+$ good basis for~$L$.

  To prove the claim, let $\pi=\pi_1(S^3\sm K)$, and consider the epimorphism
  $\phi\colon \pi\to F$ onto the free group $F$ given by the Pontryagin-Thom
  construction for~$Y$.  Since $\gamma$ is disjoint from $Y$, $\gamma$ lies in
  the kernel of~$\phi$. Since $\phi$ induces an isomorphism $\pi/\pi_k\cong
  F/F_k$ between the lower central series quotients for every~$k$ by Stallings'
  theorem~\cite{Stallings:1965-1}, $\gamma$~is trivial in~$\pi/\pi_k$.  It
  follows that $\overline\mu_{K\cup \gamma}(j_1\cdots j_k)=0$ whenever
  $j_1,\ldots,j_{k-1}$ correspond to components of $K$ and $j_k$ corresponds to
  the component~$\gamma$. Also, if all of $j_1,\ldots,j_k$ correspond to
  components of $K$, then $\overline\mu_{K\cup \gamma}(j_1\cdots j_k) =
  \overline\mu_{K}(j_1\cdots j_k) = 0$ since $K$ is a boundary link. The
  vanishing of these Milnor's invariants, together with cyclic symmetry, in
  particular implies the vanishing of Milnor's invariants of $K\cup\gamma$ for
  all non-repeating multi-indices. Thus $K\cup \gamma$ is homotopically
  trivial~\cite{Milnor:1954-1}. This completes the proof of the claim and
  therefore of the lemma.
\end{proof}

\subsection{Examples that are not Whitehead doubles}
\label{subsection:examples}

In this section we present the promised examples of links that can be freely
sliced using Theorem~\@\ref{theorem:main}, but which are neither Whitehead
doubles nor $\partial^2$-links.

Let $\beta$ be a 2-component string link whose closure $\widehat\beta$ is a
non-slice link with trivial linking number.  For instance, $\beta$ could be the
string link in Figure~\ref{figure:wh-string-link}, whose closure is the
Whitehead link. Let $\beta_{(k,\ell)}$ be the $(k+\ell)$-component string link
obtained by replacing the first and second strand of $\beta$ with their
untwisted $k$ and $\ell$ cables respectively.  Let $L(\beta)$ be the
$2$-component link shown in Figure~\ref{figure:non-wh-example}.

\begin{figure}[ht!]
  $\beta=\vcenter{\hbox{\includegraphics[scale=.8]{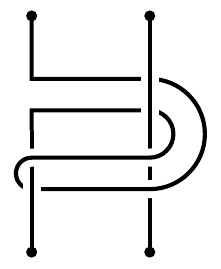}}}$
  \caption{A string link $\beta$ whose closure is the Whitehead link.}
  \label{figure:wh-string-link}
\end{figure}
\begin{figure}[ht!]
  \includegraphics[scale=.9]{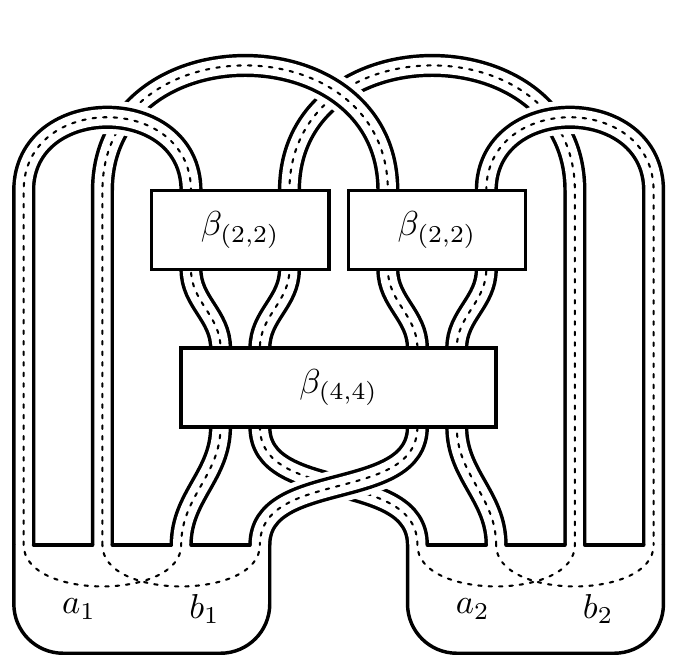}
  \caption{The link~$L(\beta)$.  The boxes indicate the relevant string links
   arising as cables of~$\beta$.}
  \label{figure:non-wh-example}
\end{figure}

\begin{corollary}
  \label{corollary:example}
  The link $L(\beta)$ is freely slice if $\widehat\beta$ has trivial linking
  number.
\end{corollary}

\begin{proof}
  Recall that a 2-component link with trivial linking number is homotopically
  trivial$^+$.  (See the discussion between Lemmas~\ref{lemma:wh-and-ht+}
  and~\ref{lemma:b^2-links-have-ht+-good-basis}.)    So, in our case,
  $\widehat{\beta_{(1,2)}}$ and $\widehat{\beta_{(2,1)}}$ are homotopically
  trivial.

  Figure ~\ref{figure:non-wh-example} is drawn in such a way that the Seifert
  surface ought to be obvious.  Consider the good basis $\{a_i,b_i\}$ shown, and
  let $K=b_1' \cup b_2'$, where $b_i'$ is a transport of $b_i$ along a normal
  direction of the surface such that $\lk(a_i,b_i')=0$.  Then $K=\widehat\beta$.
  Thus $K\cup b_1=\widehat{\beta_{(2,1)}}$ and $K\cup
  b_2=\widehat{\beta_{(1,2)}}$ are homotopically trivial.  For $i=1,2$, the link
  $K\cup a_i$ is the closure of the product of two 3-component string links $1_1
  \otimes \beta$ and $\beta_{(1,2)}$, where $1_1\otimes \beta$ denotes the split
  union of a 1-component trivial string link and~$\beta$.  Since
  $\widehat{\beta_{(1,2)}}$ and $\widehat\beta$ are homotopically trivial, it
  follows that $K\cup a_i$ is homotopically trivial. Therefore the good basis
  $\{a_i,b_i\}$ is homotopically trivial$^+$. Thus by
  Theorem~\ref{theorem:main}, the link $L(\beta)$ is freely slice.
\end{proof}

When $\beta$ is not slice, we cannot see any other way to show that $L(\beta)$
is slice. As an explicit example, let $\beta$ be the Whitehead string link in
Figure~\ref{figure:wh-string-link}.  Attempts to produce a slicing disc by
cutting 1-handles of the given Seifert surface fail.  Both components of the
link $L(\beta)$ in Figure~\ref{figure:non-wh-example} are hyperbolic knots (the
same hyperbolic knot with volume $13.16319$, in fact), as we verified using
SnapPy. Note that by contrast the components of a Whitehead double are
non-hyperbolic knots. Thus our link does not arise as a Whitehead double. Also
$L(\beta)$ bounds a collection of Seifert surfaces whose metabolisers are
realised by non-boundary links. In fact each possible metaboliser $(a_1,a_2)$,
$(a_1,b_2)$, $(b_1,a_2)$ or $(b_1,b_2)$ is a Whitehead link.  So $L(\beta)$ does
not seem to be a $\partial^2$-link. Nevertheless our Theorem~\ref{theorem:main}
applies to prove that $L(\beta)$ is slice, as shown in
Corollary~\ref{corollary:example}.

\section{Questions}\label{section-questions}

Here are some potentially interesting questions raised by the investigation in
this article.

\begin{enumerate}[label=(\arabic*)]
  \item\label{item:question-alternative-ht+-good-basis-for-wh}
  Is there a link that is not homotopically trivial$^+$ in the sense of
  Freedman and Teichner, but whose Whitehead double admits a homotopically
  trivial$^+$ good basis?
  \item\label{item:question-every-good-boundary-link}
  Does every good boundary link have a homotopically trivial$^+$ good
  basis?
\end{enumerate}

Put differently, \ref{item:question-alternative-ht+-good-basis-for-wh} asks
whether Theorem~\ref{theorem:main} can be used to slice Whitehead doubles to
which the Freedman-Teichner result does not apply.
If~\ref{item:question-alternative-ht+-good-basis-for-wh} has an affirmative
answer, the desired homotopically trivial$^+$ good basis will be on
a non-standard Seifert surface for the Whitehead double. One might ask a
generalised question: for two good bases for the same good boundary link,
possibly on different Seifert surfaces, is one homotopically trivial$^+$ if and
only if so is the other?  This has a negative answer, since one can stabilise a
Seifert surface by a genus three Seifert surface for the unknot, and apply the
technique of~\cite{Park:2016-1}.  So one should perhaps refine this version, for
example by restricting to the case that all connected components of the Seifert
surfaces have genus one.

If~\ref{item:question-every-good-boundary-link} has an affirmative answer, then
all good boundary links will be freely slice by Theorem~\ref{theorem:main}, and
consequently topological surgery would work in dimension~4 for arbitrary
fundamental groups. A Whitehead double of the Borromean rings might provide a
counterexample to~\ref{item:question-every-good-boundary-link}.

\bibliographystyle{amsalphanobts}
\def\MR#1{}
\bibliography{research}
\end{document}